\documentclass[12pt]{amsart}
\usepackage[margin=1in,includeheadfoot]{geometry}
\usepackage{enumerate}
\usepackage{comment}
\usepackage{bm}
\usepackage{tikz-cd}
\usepackage{stmaryrd}
\usepackage{amsfonts}
\usepackage{amssymb}
\usepackage{amsmath,amsthm}
\usepackage{latexsym}
\usepackage{amscd}
\usepackage{esint}
\usepackage{mathrsfs}
\usepackage{indentfirst}
\usepackage{enumitem}
\usepackage[symbol]{footmisc}
\newlist{Aenumerate}{enumerate}{1}
\setlist[Aenumerate]{label=(S\arabic*)}

\usepackage{enumerate}
\usepackage{slashed}
\usepackage{fancyhdr}
\usepackage{aliascnt}
\usepackage{pdfsync}
\usepackage{setspace}
\usepackage{color}
\usepackage{xcolor}
\definecolor{citeblue}{RGB}{0,0,139}
 \usepackage[colorlinks = true,
            linkcolor = citeblue,
            urlcolor  =  citeblue,
            citecolor =  citeblue,
            anchorcolor = citeblue]{hyperref}
\newtheorem{theorem}{Theorem}[section]

\newtheorem{lemma}[theorem]{Lemma}
\newtheorem{prop}[theorem]{Proposition}

\theoremstyle{definition}
\newtheorem{defi}[theorem]{Definition}
\theoremstyle{remark}
\newtheorem{rmk}[theorem]{Remark}
\numberwithin{equation}{section}
\numberwithin{theorem}{section}

\newcommand{\abs}[1]{\left\vert#1\right\vert}
\newcommand{\inner}[1]{\left\langle#1\right\rangle}
\newcommand{\set}[1]{\left\{#1\right\}}

\newcounter{stepnum}[section]

\newcounter{claimnum}[subsection]
\newcommand{\claim}{%
	\par
	\refstepcounter{claimnum}%
	\noindent  
	\textbf{Claim \arabic{claimnum}}.\enspace\ignorespaces
}

\newcounter{casenum}[section]

\usepackage{etoolbox}

\makeatletter
\newcommand{\reduceoperator}[1]{%
  \@for\next:=#1\do{\expandafter\reduceoperator@\expandafter{\next}}%
}
\newcommand{\reduceoperator@}[1]{%
  \csletcs{normal@#1@}{#1@}%
  \csedef{#1@}{\noexpand\reduceoperator@@\csname normal@#1@\endcsname}%
}
\newcommand{\reduceoperator@@}[1]{%
  \mathop{\mathpalette\reduceoperator@@@{#1}}%
}
\newcommand{\reduceoperator@@@}[2]{%
  \ifx#1\displaystyle\textstyle\fi#2%
}
\makeatother

\reduceoperator{bigcup,bigcap,bigotimes}

\newcommand{\p}[1]{\left(#1\right)}
\newcommand{\dbl}[1]{\widetilde{#1}}
\def\R{\mathbb R}
\def\C{\mathbb C}
\def\S{\mathbb S}

\selectfont
\allowdisplaybreaks[4]
\makeatletter
\def\@tocline#1#2#3#4#5#6#7{\relax
	\ifnum #1>\c@tocdepth 
	\else
	\par \addpenalty\@secpenalty\addvspace{#2}%
	\begingroup \hyphenpenalty\@M
	\@ifempty{#4}{%
		\@tempdima\csname r@tocindent\number#1\endcsname\relax
	}{%
		\@tempdima#4\relax
	}%
	\parindent\z@ \leftskip#3\relax \advance\leftskip\@tempdima\relax
	\rightskip\@pnumwidth plus4em \parfillskip-\@pnumwidth
	#5\leavevmode\hskip-\@tempdima
	\ifcase #1
	\or\or \hskip 2em \vspace{0.5mm}\or \hskip 3em \else \hskip 3em \fi%
	#6\nobreak\relax
	\hfill\hbox to\@pnumwidth{\@tocpagenum{#7}}\par
	\nobreak
	\endgroup
	\fi}
\makeatother

\begin{document}
	\title[Parallel Mean Curvature Surfaces]{Parallel mean curvature surfaces with constant contact angle along free boundaries}

	\author[Gao]{Rui Gao}
	\address{School of Mathematical Sciences, Shanghai Jiao Tong University\\ 800 Dongchuan Road \\ Shanghai, 200240 \\ P. R. China}%
	\email{gaorui0416@sjtu.edu.cn}
	
	\author[Zhu]{Miaomiao Zhu}
	\address{School of Mathematical Sciences, Shanghai Jiao Tong University\\ 800 Dongchuan Road \\ Shanghai, 200240 \\ P. R. China}
	\email{mizhu@sjtu.edu.cn}
	
\thanks{The first author would like to thank Professor Martin Man-chun Li, Professor Gaoming Wang and Professor Jintian Zhu for valuable discussions and suggestions. The second author would like to thank Professor Jonathan J. Zhu for useful comments.}
	 \date{\today}
	 \dedicatory{}
	
\begin{abstract}
    We classify branched immersed disks in space forms with non-zero parallel mean curvature vector and non-orthogonal constant contact angle along the boundary in 4-dimensional space form. For higher codimensional case, we prove a codimension reduction theorem for branched immersed bordered Riemann surfaces of higher genus with multiple boundary components under the same parallel mean curvature and constant contact angle assumptions. Furthermore, we construct a family of explicit examples of branched minimal immersions satisfying the non-orthonormal constant contact angle free boundary condition, which demonstrate the sharpness of both the classification result and the codimension reduction result.
\end{abstract}

\maketitle
\vskip2cm

\section{Introduction}\label{intro}

\subsection{Backgrounds}   
\ 
\vskip5pt

The study of surfaces in 3-dimensional Riemannian manifolds is a cornerstone of differential geometry. Among the various geometrically distinguished families of surfaces, those of \textit{constant mean curvature} (CMC) are of fundamental importance. A natural generalization to higher codimensions is the class of surfaces whose mean curvature vector is parallel in the normal bundle; these are known as \textit{parallel mean curvature} (PMC) surfaces.

This paper is devoted to investigating the rigidity and classification of PMC surfaces branched immersed in a geodesic ball \( B^n \) of a real space form \( \mathbb{R}^n(c) \) for some $c \in \R$. We focus on the case where these surfaces have non-empty boundary lying on the sphere \( \partial B \), and where the branched immersion maintains a constant contact angle along this boundary (see Definition \ref{defi:contact angle} below for precise descriptions). 

The theory of closed surfaces with parallel mean curvature vector has been studied in the classical works of Schouten-Struik \cite{Schouten-Struik1938}, Coburn \cite{Coburn1939}, and Wong \cite{Wong1946} around the 1940s. A systematic investigation of their rigidity and classification theory in homogeneous manifolds began in the early 1970s. Seminal contributions from this period include Ferus's characterization of spherical immersions in 4-dimensional space forms \cite{Ferus1971}, and Yau's rigidity results for submanifolds in spaces of constant curvature, which encompassed a classification in 4-dimensional real space forms \cite{yau1974}. Independent and related work was also conducted by Hoffman \cite{Hoffman-thesis1972, hoffman1973jdg} and Chen \cite{chen-Chern72}. This line of research has been extended to various types of ambient geometries, see for instances \cite{Kenmotsu-Zhou2000, DeLira-Vitorio2010, Alencar-doCarmo-Tribuzy2010,Fetcu2012,Torralbo-Urbano2012,
Fetcu-Rosenberg2015} and references therein. 

While the classification of closed PMC surfaces in space forms $\mathbb{R}^n(c)$ has been extensively studied, the corresponding results for bordered PMC surfaces with free boundary constraints are less understood.  For branched immersion with parallel mean curvature disk into a ball $B^n$ in a space form $\R^n(c)$ that meets the boundary $\partial B^n$ orthogonally, Fraser-Schoen \cite{Fraser-Schoen2015} showed that such a surface must be contained in a 3-dimensional totally geodesic submanifold of $\R^n(c)$ and is totally umbilic. This extends the results of Nitsche \cite{Nitsche1985} and Souam \cite{Souam1997} in dimension 3 to higher codimensions. When the bordered surface meets $\partial B^n$ at a constant angle along the free boundary, in the codimensional one case ($n = 3$), Ros-Souam \cite{Ros-Souam1997} proved that such a surface must be totally umbilic if it has genus zero, see also Chodosh-Edelen-Li \cite{Chodosh-Edelen-Li} for rigidity of capillary minimal discs in a hemisphere $\mathbb{S}^3_+$. Very recently, Naff-Zhu \cite{naff2025freeboundarycapillaryminimal} showed that any immersed PMC disk is totally umbilic under some ``parallel projected conormal" condition along the free boundary $u(\partial D) \subset \partial B^n$.

The major work of this paper is the generalization of the rigidity results of \cite{Nitsche1985,Souam1997,Ros-Souam1997,Fraser-Schoen2015,naff2025freeboundarycapillaryminimal} to the setting of branched immersions from a disk into space forms with non-zero parallel mean curvature vector field and constant non-zero contact angle. Furthermore, we provide a codimension reduction result for surfaces of higher genus and  multiple boundary components.

\subsection{Main Results}\ 
\vspace{5pt}

We begin by analyzing the simpler case of codimension two. In this case, we establish the following complete classification for PMC surface with non-zero constant contact angle along the free boundary.

\begin{theorem}\label{main theorem 1}
Let $u : D \rightarrow B^4$ be a branched immersion with non-zero parallel mean curvature vector from the unit disk $D \subset \R^2$ into a geodesic ball $B^4$ of a $4$-dimensional space form $\R^4(c)$ of constant curvature $c \in \R$ such that $u(D)$ meets $\partial B^4$ at a constant non-zero contact angle $\theta$\footnote{See Definition \ref{defi:contact angle} and Figure \ref{fig:capillary_angle} for the precise description of the contact angle $\theta$ between $u(D)$ and $\partial B^4$ along the free boundary $u(\partial D) \subset \partial B^4$.}. 

Then $u(D)$ is a totally umbilic surface contained in a $3$--dimensional totally umbilic submanifold $P^3\subset \R^4(c)$.
More precisely:
\begin{enumerate}
\item If $c>0$, then $P^3$ is a round $\S^3(c_1)\subset \S^4(c)$ with sectional curvature $c_1\ge c$ and
$u(D)$ is a spherical cap in $P^3$.
\item If $c=0$, then $P^3$ is either an affine $\R^3\subset\R^4$  or a round $\S^3(c_2)\subset\R^4$ with $c_2>0$, and $u(D)$ is a plane disk or a spherical cap in $P^3$.
\item If \(c < 0\), then \(P^3 \subset \mathbb{H}^4(c)\) is either totally geodesic, equidistant, a horosphere, or a geodesic sphere; correspondingly, \(u(D)\) is a piece of a totally geodesic plane disk, an equidistant surface, a horosphere, or a round sphere in \(P^3\) respectively, with constant contact angle along \(u(\partial D) \subset P^3\cap \partial B^4 \subset P^3\) in all cases.
\end{enumerate}

Furthermore, if $u: A(r,R) \rightarrow B^4$ is a branched immersion with non-zero parallel mean curvature vector from the annulus $A(r,R) \subset \R^2$, $0<r < R$, into  $B^4$ meets $\partial B^4$ admitting a constant non-zero contact angle along each boundary component of $u(\partial A(r,R))$, then $u(A(r,R))$ has no umbilic point.
\end{theorem}

For higher codimensional case $n -2 \geq 3$, we establish the following reduction of codimension result for PMC surface with higher genus and multiple boundary components.
\begin{theorem}\label{main theorem 2}
    Let \( u : \Sigma \rightarrow B^n \) be a branched immersion with non-zero parallel mean curvature vector from a bordered Riemann surface \( \Sigma \) into a geodesic ball \( B^n \) of an \( n \)-dimensional space form \( \mathbb{R}^n(c) \) of constant curvature \( c \in \mathbb{R} \), such that each component of \( u(\Sigma) \) meets \( \partial B^n \) at a constant non-zero angle $\theta_i$, $1 \leq i \leq k$, where $k$ is the number of boundary components of $\Sigma$. Then, $u(\Sigma)$ is one of the following surfaces:
    \begin{itemize}
        \item A minimal surface in a totally umbilic hypersurface $M^{n-1}$ of $\R^n(c)$ with constant contact angle $\theta_i^\prime \geq \theta_i$ along $u(\partial \Sigma) \subset \partial B^n \cap M_1$, $1 \leq i \leq k$;
        \item  A surface with constant mean curvature in a 3-dimensional totally umbilic submanifold $M_2$ of $\R^n(c)$ with constant contact angle $\dbl{\theta}_i \geq \theta_i$ along $u(\partial \Sigma) \subset \partial B^n \cap M_2$, $1 \leq i \leq k$. In particular, if $\Sigma$ has genus 1 and boundary components $k=1$, then $u(\Sigma)$ is totally umbilic and is characterized as in Theorem \ref{main theorem 1}; If $\Sigma$ has genus 1 and boundary components $k=2$, then $u(\Sigma)$ has no umbilic point.
    \end{itemize}
\end{theorem}

\subsection{Some Explicit Examples}\ 

\vskip5pt

In this subsection, we construct explicit examples of branched minimal immersions in space forms that meet the boundary of a geodesic ball at a \emph{constant non-orthogonal} contact angle. A key feature is that these immersions need not reduce in codimension, that is, their images are not contained in any proper totally umbilic submanifold of the ambient space. Moreover, as the contact angle $\theta \nearrow \frac{\pi}{2}$, these examples exhibit an asymptotic codimension reduction, that is, the image becomes asymptotically contained in a lower-dimensional totally umbilic submanifold. This is a new phenomenon in the non-orthogonal contact angle setting. When the contact angle is orthogonal, any branched minimal disk in a geodesic ball of $\R^n(c)$ is totally umbilic and contained in some $3$-dimensional totally geodesic submanifold of $\R^n(c)$, as established in \cite{Fraser-Schoen2015}. Notably, the examples in Section~\ref{Section 1.3.1} and Section~\ref{section 1.3.3} below indicate that the assumptions in Theorem~\ref{main theorem 1} cannot be weakened in general, and they realize the first alternative in Theorem~\ref{main theorem 2}.

\subsubsection{An example in  \texorpdfstring{$\R^4$}{R4}}\label{Section 1.3.1}\ 

Identify $\R^4\simeq\C^2$ and let $B^4=\{w\in\C^2:\ |w|<1\}$ with $\partial B^4=S^3$ and outward unit normal $\boldsymbol{n}(w)=w$ on $S^3$.
Fix $a,b>0$ with $a^2+b^2=1$ and define
\[
u:D\to B^4,\qquad u(z)=(a z,\; b z^2),\quad z\in D=\{|z|<1\}.
\]
Then, $u(D)\subset B^4$ and $u(\partial D)\subset \partial B^4$. Moreover, $\partial_z u=(a,2bz)\neq 0$, so $u$ is an immersion. Since $u$ is holomorphic as a $\C^2$-valued function, it is harmonic and conformal, hence $u(D)\subset\R^4$ is a minimal immersion.
Along $\partial D$ we have $\boldsymbol{n}=u$, and the outward unit conormal of $u$ is $\boldsymbol{\nu}=\partial_r u/|\partial_r u|$. Thus, the contact angle is  given by
\[
\sin\theta=\langle \boldsymbol{\nu},\boldsymbol{n}\rangle_{\R^4}
=\frac{\langle u,\partial_r u\rangle_{\R^4}}{|\partial_r u|}
=\frac{a^2+2b^2}{\sqrt{a^2+4b^2}}\in(0,1),
\]
hence, $\theta\in(0,\frac{\pi}{2})$ is non-orthogonal. Furthermore, $u(D)$ cannot be contained in any totally umbilic hypersurface of $\R^4$\footnote{Indeed, the hyperplane case is excluded by restricting $\langle u,\xi\rangle_{\R^4}$ to $\partial D$ and separating Fourier modes $e^{it}$ and $e^{i2t}$, while the sphere case is excluded since $\Delta |u-c|^2=2(|u_x|^2+|u_y|^2)>0$ for a minimal immersion $u$, so $|u-c|^2$ cannot be constant for any $c \in \R^4$.}. 
\begin{rmk}
    Observe that as \( b \searrow 0 \), we have \( \theta \nearrow \pi/2 \), and the image of \( u \) constructed above tends to a flat disk in $\R^4$. This demonstrates that the assumption of a non-vanishing mean curvature vector in Theorem \ref{main theorem 1} and Theorem \ref{main theorem 2} for a non-orthogonal constant contact angle case is essential and optimal. 
\end{rmk}
\subsubsection{A generalization to \texorpdfstring{$\R^{2m}$}{Lg}}\ 

Furthermore, the above example can be extended into the higher codimensional setting. Similarly, identify $\R^{2m}\simeq\C^m$ and let
$B^{2m}=\{w\in\C^m:\ |w|<1\}$, $\partial B^{2m}=S^{2m-1}$. Fix integers $1=k_1<k_2<\cdots<k_m$ and coefficients $c_1,\dots,c_m\in\C$ with $\sum_{j=1}^m |c_j|^2=1$ and $c_1\neq 0$, and set
\[
u:D\to B^{2m},\qquad u(z)=\big(c_1 z^{k_1},\,c_2 z^{k_2},\,\dots,\,c_m z^{k_m}\big).
\]
Then, $u(D)\subset B^{2m}$, $u(\partial D)\subset \partial B^{2m}$, and $u$ is a holomorphic immersion and therefore minimal in $\R^{2m}$.
Along $\partial D$, the unit conormal of $u$ is $\boldsymbol{\nu}=\partial_r u/|\partial_r u|$ and the contact angle is constant and non-orthogonal with
\[
 \sin\theta=\langle \boldsymbol{\nu},\boldsymbol{n}\rangle_{\R^{2m}}
=\frac{\sum_{j=1}^m k_j |c_j|^2}{\sqrt{\sum_{j=1}^m k_j^2 |c_j|^2}} \in (0,1).
\]

\subsubsection{Veronese Minimal Immersion with Non-orthonormal Contact Angle in \texorpdfstring{$\S^4$}{S4}}\label{section 1.3.3}\

Let $\S^4 \subset \R^5$ be the unit round sphere.  Recall that the Veronese immersion $u:\S^2\to \S^4$ is defined by
\[
u(x_1,x_2,x_3)=\left(\frac{\sqrt3}{2}(x_1^2-x_2^2),\ \sqrt3\,x_1x_2,\ \sqrt3\,x_1x_3,\ \sqrt3\,x_2x_3,\ \frac12(3x_3^2-1)\right),\quad x\in\S^2\subset\R^3.
\]
It is straightforward to verify that $u:\S^2 \rightarrow \S^4$ is conformal and harmonic, hence, its image is  a minimal surface in $\S^4$.

Using the spherical coordinates of $\S^2$: $x_1=\sin\varphi\cos\vartheta$, $x_2=\sin\varphi\sin\vartheta$ and $x_3=\cos\varphi$ with $(\vartheta,\varphi)\in [0,2\pi) \times [0,\pi ]$, fix any latitude $\varphi_0\in(0,\frac{\pi}{2})$ and set $t_0:=u_5(\varphi_0)$ and $\rho:=\arccos(t_0)\in(0,\pi)$.  Let $B_\rho:=\{y\in \S^4:\langle y,e_5\rangle\ge \cos\rho\}=\{y_5\ge t_0\}$ be the geodesic ball centered at $e_5$, and let $D_{\varphi_0}:=\{(\varphi,\theta):0\le \varphi\le \varphi_0\}\subset\S^2$ be the spherical cap, conformally identifying it with the unit disk. Then, we have $u(D_{\varphi_0})\subset \overline{B_\rho}$ and $u(\partial D_{\varphi_0})\subset \partial B_\rho=\{y_5=t_0\}$.  Denote by $\boldsymbol{n}$ the outward unit normal of $\partial B_\rho\subset \S^4$ and by $\boldsymbol{\nu}$ the outward unit conormal of $u(D_{\varphi_0})$ along $u(\partial D_{\varphi_0})$.  Along $\partial B_\rho$ one has
\[
\boldsymbol{n}(y)=-\frac{e_5-y_5y}{\sqrt{1-y_5^2}},\qquad y\in\partial B_\rho,
\]
 and along $\partial D_{\varphi_0}$ one has $\boldsymbol{\nu}=u_\varphi/{|u_\varphi|}={u_\varphi}/{\sqrt3}$.
Hence, the contact angle between $u( D_{\varphi_0})$ and $\partial B_\rho$ along
$\partial D_{\varphi_0}$ is given by 
\[
\sin\theta=\langle \boldsymbol{n},\boldsymbol{\nu}\rangle=\frac{\langle \boldsymbol{n},u_\varphi\rangle}{|u_\varphi|}
=\frac{2\cos\varphi_0}{\sqrt{1+3\cos^2\varphi_0}} \in (0,1),
\]
This implies, $\theta\in(0,\frac{\pi}{2})$ is constant and non-orthogonal. Besides, if we let $\theta \nearrow \pi/2$, then the image of $u(D_{\varphi_0})$ tends to a point in $\S^4$, leading to the trivial case.  Finally, $u(D_{\varphi_0})$ is not contained in any proper totally umbilic hypersurface of $\S^4$. Indeed, since totally umbilic hypersurfaces are subspheres $\S^4\cap\{ \langle v,\cdot\rangle=t\}$ for some $v\in \R^5$ and $t \in \R$, if $\langle v,u\rangle\equiv t$ on the open set $\mathring{D_{\varphi_0}}$, analyticity of $u$ forces $\langle v,u\rangle\equiv t$ on $\S^2$. However, it is well known that the  Veronese minimal immersion from $\S^2$ into $\S^4$ is linearly full which leads to a contradiction. Therefore, $u(D_{\varphi_0})$ cannot be contained in any lower-dimensional totally umbilic subsphere.
\subsection{Difficulties and Main Ideas}\ 

\vskip5pt
We plan to draw on the complex analytic strategy developed for closed surfaces in \cite{yau1974,hoffman1973jdg,chen-Chern72} and for orthogonal free boundary setting in \cite{Fraser-Schoen2015}, rooted in the classical Hopf holomorphic quadratic differential. The guiding principle is to construct an appropriate holomorphic differential form on the surface and then exploit two analytic inputs---holomorphicity in the interior and real-valuedness along the boundary---to force the differential to vanish identically by Schwarz reflection principle. This vanishing yields strong geometric rigidity, typically collapsing the traceless part of the second fundamental form in some direction and hence implying the totally umbilicity and the codimension reduction property of the surface. In the non-orthogonal constant angle setting, the boundary real-valuedness is no longer fulfilled in general as indicated by the examples in Section~\ref{section 1.3.3}, unlike in \cite{Nitsche1985,Souam1997,Ros-Souam1997,Fraser-Schoen2015,naff2025freeboundarycapillaryminimal}. Our main idea is therefore to first prove a codimension reduction result, showing that the surface must lie in an appropriate totally umbilic submanifold. Within this reduced ambient manifold, we can then construct a suitable Hopf-type holomorphic quadratic differential whose boundary behavior is sufficiently rigid, and this ultimately yields the desired classification result. A key starting point of this approach is Proposition~\ref{lem:PMC-and-angle-in-umbilical-S}, which provides the inheritance mechanism needed for codimension reduction. More precisely, once the surface is shown to lie in a totally umbilic hypersurface $M$, Proposition~\ref{lem:PMC-and-angle-in-umbilical-S} guarantees that the PMC condition descends to the lower dimensional ambient space $M$ and shows the transversal intersection of $M$ with $\partial B^n$, so that the non-zero constant contact angle condition is well defined and also descends with an explicit contact angle conversion $\theta\mapsto\theta_M$\footnote{Here, $\theta_M \geq \theta$ is the constant contact angle of $u(\Sigma)$ with $M\cap \partial B^n$ in $M$.}.

The key idea to prove Theorem~\ref{main theorem 1} is to first use codimension--two holomorphic differentials to reduce the surface into a totally umbilic $3$-submanifold $P^3\subset \R^4(c)$. Within $P^3$ the problem becomes a CMC capillary disk in a $3$-space form by Proposition~\ref{lem:PMC-and-angle-in-umbilical-S}. We then show that the constant capillary angle with a totally umbilic hypersurface induces a boundary curvature-line condition as in \cite{Ros-Souam1997}, which precisely restores the missing boundary real-valuedness for the Hopf type differential. Schwartz reflection then forces the Hopf type differential to vanish, completing the classification in Theorem~\ref{main theorem 1}.

Building on this $4$-dimensional rigidity, Theorem~\ref{main theorem 2} treats arbitrary $\R^n(c)$ by combining normal-bundle flatness in the second alternative of Theorem \ref{main theorem 2} and the construction of a parallel rank-4 subbundle via a global gluing argument to reducing the codimension so as to place the surface in a totally geodesic $\R^4(c)\subset\R^n(c)$. Then, the same $P^3$-reduction as in Theorem \ref{main theorem 1} completes the proof of Theorem \ref{main theorem 2}. At each reduction step, Proposition~\ref{lem:PMC-and-angle-in-umbilical-S} ensures that both the parallel mean curvature condition and the non-zero constant contact angle condition on every boundary component persist, allowing the complex-analytic rigidity argument to apply for a surface with higher genus and multiple boundary components lying in reduced ambient submanifold.

\subsection{Organization of the paper}\ 
\vspace{5pt}

This paper is organized as follows. In Section~\ref{sec:preliminaries}, we introduce the geometric setting, fix notation, and construct quadratic differentials associated with surfaces of parallel mean curvature under constant contact angle boundary conditions. Section~\ref{sec:dimension-reduction} establishes the key dimension reduction result, Proposition~\ref{lem:PMC-and-angle-in-umbilical-S}, by showing that the PMC condition and constant contact angle boundary condition descend to a totally umbilic hypersurface $M$. Finally, in Section~\ref{sec:proof-main}, we use this reduction tool to analyze the holomorphic quadratic differentials obtained in Section~\ref{sec:preliminaries} for a PMC surface with constant non-zero contact angle to complete the proofs of Theorem~\ref{main theorem 1} and Theorem~\ref{main theorem 2}.

\vskip1cm
\section{Preliminaries and Construction of Quadratic Differentials}\label{sec:preliminaries}

In this Section, particularly in Section \ref{sec:preliminaries 1}, we introduce the geometric setting, fix notation, and recall basic facts on surfaces with parallel mean curvature in space forms, together with the constant contact angle conditions. We then, in Section \ref{sec:preliminaries 2}, we construct the quadratic differentials associated with the trace-free parts of the second fundamental form in certain parallel normal directions.

\subsection{Settings and Basic Properties}\label{sec:preliminaries 1}\
\vspace{5pt}

Let $(\Sigma,g)$ be an oriented Riemann surface with smooth boundary $\partial \Sigma \neq 0$ and $N$ be a Riemannian $n$-manifold with metric $h = \inner{\cdot,\cdot}$ and Levi-Civita connection $\overline{\nabla}$. Suppose $\mathcal{K} \subset N$ is a closed hypersurface of $N$ with outward unit normal vector $\boldsymbol{n}$, serving as a supporting submanifold of the free boundary problem constraint.
Denote $z = x^1 + \sqrt{-1} x^2$ to be the complex coordinate of the $\Sigma$ with
\begin{equation*}
    \frac{\partial}{\partial z} = \frac{1}{2}\p{\frac{\partial}{\partial x^1} - \sqrt{-1} \frac{\partial }{\partial x^2}} \quad \text{and}\quad \frac{\partial}{\partial \overline{z}} = \frac{1}{2}\p{\frac{\partial}{\partial x^1} + \sqrt{-1} \frac{\partial }{\partial x^2}}.
\end{equation*}

Suppose $u : \Sigma \rightarrow N$ is a branched immersion (i.e., weakly conformal and immersed except finitely many branch points) in $N$ with free boundary constraint $u(\partial \Sigma) \subset \mathcal{K}$. Under the above complex coordinate, we have
\begin{equation}
    \label{eq:complex eq 1}
    \inner{\frac{\partial u}{\partial z}, \frac{\partial u}{\partial \overline{z}}} = \frac{1}{2} \lambda^2 \quad \text{and}\quad \inner{\frac{\partial u}{\partial z}, \frac{\partial u}{\partial {z}}} =0,
\end{equation}
where $\lambda^2$ is the conformal factor determined by
\begin{equation}\label{eq:complex metric 1}
    u^*h = \lambda^2 g,
\end{equation}
for some smooth function $\lambda \geq 0$ vanishing at branch points. 
At each non-branch point $p \in \Sigma$, $T_{u(p)}N$ admits an orthogonal decomposition $T_{u(p)}N = du(T_{p} \Sigma)\oplus du(T_{p} \Sigma)^\perp$. Around each branch point $p_0 \in \Sigma$, the tangent space $du(T_{p} \Sigma)$ and normal space $du(T_{p} \Sigma)^\perp$ extend across $p_0$ smoothly to obtain the rank-two and rank-$(n-2)$ vector bundles over the whole $\Sigma$, also denoted by $du(T \Sigma)$ and $du(T \Sigma)^\perp$, respectively.  According to this extended decomposition, we denote $X^\top \in du(T_p \Sigma)$ and $X^\perp \in du(T_p \Sigma)^\perp$ as the components of vector $X \in T_{u(p)}N$, and denote $\nabla$ to be the connection on $du(T \Sigma)$ induced from $\overline{\nabla}$, which is the Levi-Civita connection of the induced metric $u^*h$ on $\Sigma$.

Given a normal vector field ${\boldsymbol{\xi}} \in \Gamma(du(T \Sigma)^\perp)$, the shape operator $A_{\boldsymbol{\xi}}$ of $u(\Sigma)$ associated to ${\boldsymbol{\xi}}$ is a self-adjoint linear transformation $A_{\boldsymbol{\xi}}:du(T \Sigma) \rightarrow du(T \Sigma)$, defined by 
$$A_{\boldsymbol{\xi}} (X) = - \p{\overline{\nabla}_X {\boldsymbol{\xi}}}^\top.$$
And the second fundamental form $A:du(T \Sigma)\times du(T \Sigma) \rightarrow du(T_{p} \Sigma)^\perp$ of $u(\Sigma)$ in $N$ satisfies $\inner{A(X,Y),{\boldsymbol{\xi}}} = \inner{A_{\boldsymbol{\xi}}(X), Y}$ for all $X,\, Y \in \Gamma(du(T \Sigma))$. Choosing an orthonormal basis $\set{e_1, e_2}$ of $du(T\Sigma)$, the mean curvature vector $H = H(u)$ of $u$ is defined as $H = 1/2 ( A(e_1,e_1) + A(e_2,e_2))$.  The normal connection $\nabla^\perp : du(T \Sigma) \times du(T \Sigma)^\perp \rightarrow du(T \Sigma)^\perp$ is defined as $\nabla^\perp_X {\boldsymbol{\xi}} : = \p{\overline{\nabla}_X {\boldsymbol{\xi}}}^\perp$ for all $X \in \Gamma(du(T \Sigma))$ and ${\boldsymbol{\xi}} \in \Gamma(du(T_ \Sigma)^\perp)$.

Equipped with these notations, we can give the precise definition of a branched immersed surface with a non-zero parallel mean curvature vector and a constant contact angle along the free boundary (see also {\cite[Remark 2.2]{naff2025freeboundarycapillaryminimal}}).

\begin{defi}[PMC branched immersion with constant contact angle]\label{defi:contact angle}
 A  branched immersion $u:\Sigma\to N$ with free boundary constraint $u(\partial\Sigma)\subset \mathcal{K}$ is said to have
\emph{parallel mean curvature} (PMC) if its mean curvature vector field
\[
H \in \Gamma\big(du(T\Sigma)^\perp\big)
\]
is parallel with respect to the normal connection, i.e.
\[
\nabla^\perp H \equiv 0
\quad\text{on } \Sigma\setminus \mathcal B,
\]
where $\mathcal B$ denotes the discrete branch set.

Along $\partial\Sigma\setminus \mathcal B$, let $\boldsymbol{\tau}$ be the unit tangent vector field of the boundary curve
$u(\partial\Sigma)$ compatible with its orientation, and let $\boldsymbol{\nu}$ be the outward unit conormal of
$u(\partial\Sigma)\subset u(\Sigma)$, so that $\boldsymbol{\tau},\boldsymbol{\nu}\in \Gamma\big(du(T\Sigma)|_{u(\partial\Sigma)}\big)$
and $\langle \boldsymbol{\tau},\boldsymbol{\nu}\rangle_{TN}=0$.
For each non-branch point $p\in \partial\Sigma$, define
\[
\cos\theta(p)\;:=\;\max\Big\{\,\langle \boldsymbol{x},\boldsymbol{\nu}\rangle_{T_{u(p)}N}\;:\;
\boldsymbol{x}\in T_{u(p)}\mathcal{K}\cap \mathbf{Span}\{\boldsymbol{\tau}\}^\perp,\ \ |\boldsymbol{x}|_{T_{u(p)}N}=1\,\Big\}
\in[0,1],
\]
and let $\theta(p)\in[0,\pi/2]$ be determined by $\cos\theta(p)$.
We call $\theta(p)$ the \emph{contact angle} of $u(\Sigma)$ with $M$ at $p$. 
If $\theta$ is constant along $\partial\Sigma$ after extending $\theta$ smoothly to boundary branch points, we say that $u$ has \emph{constant contact angle} along $\partial\Sigma$.

Moreover, in the non-orthogonal case $\cos\theta(p)>0$, equivalently $\operatorname{proj}_{T_{u(p)}\mathcal{K}}\boldsymbol{\nu}\neq 0$, the maximizer is unique and is given by
\[
\boldsymbol{\mu}(p)
=\frac{\operatorname{proj}_{T_{u(p)}\mathcal{K}}\boldsymbol{\nu}}{\bigl|\operatorname{proj}_{T_{u(p)}\mathcal{K}}\boldsymbol{\nu}\bigr|}
\in T_{u(p)}\mathcal{K}\cap \mathbf{Span}\{\boldsymbol{\tau}\}^\perp,
\qquad
\cos\theta(p)=\langle \boldsymbol{\mu}(p),\boldsymbol{\nu}(p)\rangle_{T_{u(p)}N}.
\]
In the orthogonal case $\cos\theta(p)=0$, i.e.\ $\boldsymbol{\nu}\perp T_{u(p)}\mathcal{K}$, the contact angle $\theta(p)=\pi/2$ is still well-defined, while the maximizer is not unique.
\end{defi}
\begin{figure}[h!]
    \centering
    \vspace{0cm} 
    \includegraphics[width=0.4\textwidth]{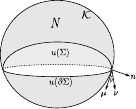} \hspace{0.5cm}
    \includegraphics[width=0.47\textwidth]{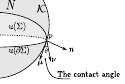}
    \vspace{0cm} 
    \caption{The contact angle between $u(\Sigma)$ and $\mathcal{K}$ along the boundary $u(\partial \Sigma)$}
    \label{fig:capillary_angle}
\end{figure}
\begin{rmk}\label{rmk 1}
    Note that by the choice of $\boldsymbol{\mu}$, we have 
    \begin{equation}
        \boldsymbol{\nu} \perp \boldsymbol{x} \quad \text{for all } \boldsymbol{x} \in T_{\boldsymbol{\mu}} := \set{\boldsymbol{x}\in T\mathcal{K}: \inner{\boldsymbol{x},\boldsymbol{\mu}} = 0}.
    \end{equation}
    Since $\dim(T_{\boldsymbol{\mu}}) = n-2$, it implies that $\boldsymbol{\mu}$, $\boldsymbol{\nu}$ and $\boldsymbol{n}$ are linearly dependent, that is, they are lie on a two dimensional plane such that 
    \begin{equation}\label{eq: decom 1}
        \boldsymbol{\nu} = \cos\theta\, \boldsymbol{\mu} + \sin\theta \,\boldsymbol{n}.
    \end{equation}
    In particular, when $\theta = \pi/2$, we have $\boldsymbol{\nu} = \boldsymbol{n}$, which corresponds to the case where $u(\Sigma)$ meets $\mathcal{K}$ orthogonally along $u(\partial \Sigma)$.
\end{rmk}
Recall that the normal curvature tensor $R^\perp$ of the normal connection $\nabla^\perp$ on the normal bundle $du(T_{p} \Sigma)^\perp$ is defined by
\begin{equation}\label{eq:defi of normal curvature}
    R^\perp(X,Y){\boldsymbol{\xi}} : = \nabla^\perp_X \nabla^\perp _Y {\boldsymbol{\xi}} - \nabla^\perp_Y \nabla^\perp_X {\boldsymbol{\xi}} - \nabla^\perp_{[X,Y]} {\boldsymbol{\xi}}
\end{equation}
and
\begin{equation}
    R^\perp(X,Y;{\boldsymbol{\xi}},{\boldsymbol{\eta}}) : = \inner{R^\perp(X,Y){\boldsymbol{\xi}},{\boldsymbol{\eta}}}
\end{equation}
for $X,\, Y \in \Gamma(du(T \Sigma))$ and ${\boldsymbol{\xi}},\, {\boldsymbol{\eta}} \in \Gamma(du(T \Sigma)^\perp)$. In particular, if the ambient manifold $N$ has constant sectional curvature $c \in \R$, then, for $X,\, Y,\, Z \in \Gamma(du(T \Sigma))$ and ${\boldsymbol{\xi}},\, {\boldsymbol{\eta}} \in \Gamma(du(T_{p} \Sigma)^\perp)$, the Ricci equation and the Codazzi equation for $u : \Sigma \rightarrow N$ can be written as 
\begin{equation}
    \label{eq:Ricci}
    R^\perp(X,Y;{\boldsymbol{\xi}},{\boldsymbol{\eta}}) = \inner{[A_{\boldsymbol{\xi}},A_{\boldsymbol{\eta}}]X,Y}
\end{equation}
and 
\begin{equation}
    \label{eq:Codazzi}
    (\overline{\nabla}_X A)(Y,Z) -  (\overline{\nabla}_Y A)(X,Z) = 0,
\end{equation}
respectively. Some basic properties for the branched immersion $u$ are summarized in the following Lemma.
\begin{lemma}\label{lem:properties H}
    Let $u : \Sigma \rightarrow N$ be a branched immersion with parallel mean curvature vector $H \neq 
    0$. Then, $H$ has constant length. Furthermore, if $\dim(N) = 4$, then the following holds
    \begin{enumerate}
        \item\label{lem:properties H item 1} There exists a non-zero parallel normal vector field $H^\prime \in \Gamma(du(T_{p} \Sigma)^\perp)$ such that 
        $$\set{H/|H|, H^\prime/|H|}$$
        is a global orthonormal frame in $du(T \Sigma)^\perp$ away from branch points.
        \item\label{lem:properties H item 2} The normal bundle $du(T_{p} \Sigma)^\perp$ is flat, that is, $R^\perp  \equiv 0$.
    \end{enumerate}
\end{lemma}
\begin{proof}
    The constancy of $\abs{H}$ follows from that 
    \begin{equation}
        \nabla_X \inner{H,H} = 2 \inner{\nabla^\perp H,H} = 0, \quad \text{for all } X \in \Gamma(du(T \Sigma)).
    \end{equation}
    If $\dim(N) = 4$, then the normal space of $u : \Sigma \rightarrow N$ has dimension 2, which enable us to find a continuous vector field $H^\prime$ via rotating $H$ by the angle $\pi/2$ according to the orientation of $du(T_{p} \Sigma)^\perp$. Since $H$ is parallel, $H^\prime$ is also parallel and $\abs{H^\prime} = \abs{H}$, hence $\set{H/|H|, H^\prime/|H|}$ is a global orthonormal frame in $du(T_{p} \Sigma)^\perp$ away from branch points. The part \eqref{lem:properties H item 2} follows from part \eqref{lem:properties H item 1} and the definition of normal curvature $R^\perp$ in \eqref{eq:defi of normal curvature}.
\end{proof}
\subsection{Construction of Quadratic Differentials}\label{sec:preliminaries 2}\
\vspace{5pt}

In this subsection, we present the construction of a series of quadratic differentials on $\Sigma$ and prove the holomorphicity of them. Before displaying the detailed construction, we first establish the following lemma. 
\begin{lemma}
    \label{lem:holomorphic}
    Let $N = \R^n(c)$ be an $n$-dimensional space form of sectional curvature $c \in \R$ and $u : \Sigma \rightarrow \R^n(c)$ be a branched immersion with parallel mean curvature vector $H \neq 
    0$. Then, working on $\Sigma\setminus \mathcal{B}$ away from branch points,  we have 
    \begin{equation}\label{eq0: lem holomorphic}
        \frac{\partial}{\partial \overline{z}} \inner{A\p{\frac{\partial u}{\partial z}, \frac{\partial u}{\partial z}}, {\boldsymbol{\xi}}} = 0,
    \end{equation}
    for any parallel normal section ${\boldsymbol{\xi}} \in \Gamma(du(T \Sigma)^\perp)$.
\end{lemma}
\begin{proof}
    Let ${\boldsymbol{\xi}}$ be a parallel normal section, direct computation yields
\begin{align}\label{eq1: lem holomorphic}
    \frac{\partial}{\partial \overline{z}} \inner{A\p{\frac{\partial u}{\partial z}, \frac{\partial u}{\partial z}}, {\boldsymbol{\xi}}} &= \inner{ \nabla_{\frac{\partial u}{\partial \bar{z}}}^\perp A\p{\frac{\partial u}{\partial z}, \frac{\partial u}{\partial {z}}}, {\boldsymbol{\xi}} } + \inner{A\p{\frac{\partial u}{\partial z}, \frac{\partial u}{\partial {z}}}, \nabla_{\frac{\partial u}{\partial \bar{z}}}^\perp {\boldsymbol{\xi}} }\nonumber \\
&= \inner{\nabla_{\frac{\partial u}{\partial \bar{z}}}^\perp A\p{\frac{\partial u}{\partial z}, \frac{\partial u}{\partial {z}}}, {\boldsymbol{\xi}} }.
\end{align}
The Gauss formula gives
\begin{equation}\label{eq:Gauss-complex}
\overline{\nabla}_{\frac{\partial u}{\partial \bar{z}}}\frac{\partial u}{\partial z} = du\big(\nabla_{\bar z}\partial_z\big)+A(\partial_{\bar z},\partial_z),
\qquad
\overline{\nabla}_{\frac{\partial u}{\partial z}}\frac{\partial u}{\partial \bar{z}} = du\big(\nabla_{z}\partial_{\bar z}\big)+A(\partial_{z},\partial_{\bar z}).
\end{equation}
In conformal complex coordinates for $g$ one has $\nabla_{\bar z}\partial_z=\nabla_{z}\partial_{\bar z}=0$,
hence the tangential parts vanish:
\begin{equation}\label{eq:torsionfree}
\nabla_{\frac{\partial u}{\partial \bar{z}}}\frac{\partial u}{\partial z} = 0,
\qquad
\nabla_{\frac{\partial u}{\partial z}}\frac{\partial u}{\partial \bar{z}} = 0 .
\end{equation}
Moreover, by weak conformality of $u$, we have
\begin{equation}\label{eq:A-zbarz-H}
A\left(\frac{\partial u}{\partial \bar{z}},\frac{\partial u}{\partial z}\right)=\frac18|\nabla u|^2\,H .
\end{equation}
Utilizing \eqref{eq:torsionfree}, the last term in \eqref{eq1: lem holomorphic} can be simplified to 
\begin{align}\label{eq2: lem holomorphic}
\nabla_{\frac{\partial u}{\partial \bar{z}}}^\perp A\p{\frac{\partial u}{\partial z}, \frac{\partial u}{\partial {z}}} &= (\nabla^\perp_{\frac{\partial u}{\partial \bar{z}}} A)\p{\frac{\partial u}{\partial z}, \frac{\partial u}{\partial {z}}} +2 A\p{\nabla_{\frac{\partial u}{\partial \bar{z}}} \frac{\partial u}{\partial z}, \frac{\partial u}{\partial {z}}}\nonumber  \\
&= (\nabla^\perp_{\frac{\partial u}{\partial \bar{z}}} A)\p{\frac{\partial u}{\partial z}, \frac{\partial u}{\partial {z}}}.
\end{align}
Writing the Codazzi equation \eqref{eq:Codazzi} in space form under the complex coordinates, we have
\begin{equation*}
    (\nabla^\perp_{\frac{\partial u}{\partial \bar{z}}} A)\p{\frac{\partial u}{\partial z}, \frac{\partial u}{\partial {z}}} = (\nabla^\perp_{\frac{\partial u}{\partial z}} A)\p{\frac{\partial u}{\partial \bar{z}}, \frac{\partial u}{\partial {z}}}.
\end{equation*}
Substituting it into the last term in \eqref{eq2: lem holomorphic} and recalling \eqref{eq:torsionfree} and \eqref{eq:A-zbarz-H}
yield 
\begin{align}\label{eq3: lem holomorphic}
(\nabla^\perp_{\frac{\partial u}{\partial \bar{z}}} A)\p{\frac{\partial u}{\partial z}, \frac{\partial u}{\partial {z}}} &= (\nabla^\perp_{\frac{\partial u}{\partial z}} A)\p{\frac{\partial u}{\partial \bar{z}}, \frac{\partial u}{\partial {z}}}\nonumber\\
&= {\nabla}^\perp_{\frac{\partial u}{\partial z}}\p{ A\p{\frac{\partial u}{\partial \bar{z}}, \frac{\partial u}{\partial {z}}}} - A\p{{\nabla}_{\frac{\partial u}{\partial z}} \frac{\partial u}{\partial \bar{z}}, \frac{\partial u}{\partial {z}}} - A\p{\frac{\partial u}{\partial \bar{z}}, {\nabla}_{\frac{\partial u}{\partial z}}\frac{\partial u}{\partial {z}}}\nonumber\\
& = \nabla^\perp_{\frac{\partial u}{\partial z}} \p{\frac{1}{8} \abs{\nabla u}^2 H} - \frac{1}{4}\frac{\partial \log \abs{\nabla u} }{\partial z} \abs{\nabla u}^2 H\nonumber\\
& = 0,
\end{align}
where in the third equality we used the identity  
\begin{equation}\label{eq4: lem holomorphic}
    {\nabla}_{\frac{\partial u}{\partial z}}\frac{\partial u}{\partial {z}} = 2 \frac{\partial \log \abs{\nabla u} }{\partial z} \frac{\partial u}{\partial z}.
\end{equation}
To see \eqref{eq4: lem holomorphic}, we write 
\begin{equation}\label{eq5: lem holomorphic}
    {\nabla}_{\frac{\partial u}{\partial z}}\frac{\partial u}{\partial {z}} = f_1 \frac{\partial u}{\partial {z}} + f_2 \frac{\partial u}{\partial \bar{z}}
\end{equation}
for some coefficients $f_1,\,f_2$ to be determined. By \eqref{eq:complex eq 1}, \eqref{eq:torsionfree} and \eqref{eq:A-zbarz-H}, taking the inner product with $\partial_{\bar{z}} u$ in \eqref{eq5: lem holomorphic} yields
\begin{equation*}
    f_1 \p{\frac{1}{2} \abs{\nabla u}^2} = \inner{{\nabla}_{\frac{\partial u}{\partial z}}\frac{\partial u}{\partial {z}}, \frac{\partial u}{\partial \bar{z}}} = \nabla_{\frac{\partial }{\partial {z}}} \inner{\frac{\partial u}{\partial {z}}, \frac{\partial u}{\partial \bar{z}}} = \frac{1}{2} \frac{\partial}{\partial z} \p{\abs{\nabla u}^2},
\end{equation*}
that is,
\begin{equation}\label{eq6: lem holomorphic}
    f_1 = 2 \frac{\partial \log \abs{\nabla u}}{\partial z}.
\end{equation}
Similarly, by \eqref{eq:complex eq 1}, taking the inner product with $\partial_{z} u$ in \eqref{eq5: lem holomorphic} yields
\begin{equation*}
    f_2  \p{\frac{1}{2} \abs{\nabla u}^2} = \inner{{\nabla}_{\frac{\partial u}{\partial z}}\frac{\partial u}{\partial {z}}, \frac{\partial u}{\partial z}} = \frac{\partial}{\partial z} \inner{\frac{\partial u}{\partial {z}}, \frac{\partial u}{\partial z}} = 0,
\end{equation*}
that is, $f_2 = 0$. combining this with \eqref{eq6: lem holomorphic}, \eqref{eq4: lem holomorphic} follows. In conclusion, plugging \eqref{eq3: lem holomorphic} into \eqref{eq1: lem holomorphic} leads to the \eqref{eq0: lem holomorphic}, as asserted in Lemma \ref{lem:holomorphic}.
\end{proof}

Now, let $N = \R^4(c)$ be the four dimensional space form of constant sectional curvature $c \in \R$. In the isothermal coordinate of $\Sigma \setminus \mathcal{B}$, we define the following two quadratic differentials 
\begin{equation}\label{eq:defi holomorphic 1}
    \begin{aligned}
        \Phi_H &:= \inner{A\p{\frac{\partial u}{\partial z}, \frac{\partial u}{\partial z}}, H} dz\otimes dz = \varphi(H) dz\otimes dz,\\
        \Phi_{H^\prime} &:=\inner{A\p{\frac{\partial u}{\partial z}, \frac{\partial u}{\partial z}}, H^\prime } dz \otimes dz = \varphi(H^\prime)dz \otimes dz,
    \end{aligned}
\end{equation}
where $H^\prime$ is defined in Lemma \ref{lem:properties H}. Indeed, the expressions \eqref{eq:defi holomorphic 1} of $\Phi_H$ and $\Phi_{H^\prime}$ do not depend on the choice of isothermal coordinates, hence they are well-defined globally on $\Sigma$ after extending them smoothly over the branch points. By the parallelism of $H$ and $H^\prime$, Lemma \ref{lem:holomorphic} tells us that $\Phi_H$ and $\Phi_{H^\prime}$ are both holomorphic on $\Sigma$.

Furthermore, for any linear combination ${\boldsymbol{\eta}} = a H + b H^\prime$ of $H$ and $H^\prime$ with constant coefficients $a,\, b \in \R$, the quadratic differential 
\begin{equation}\label{eq:defi holomorphic 2}
    \Phi_{\boldsymbol{\eta}} := \inner{A\p{\frac{\partial u}{\partial z}, \frac{\partial u}{\partial z}}, {\boldsymbol{\eta}}} dz\otimes dz = a \Phi_H + b \Phi_{H^\prime}
\end{equation}
is also holomorphic on $\Sigma$.
\vskip1cm

\section{A Dimension Reduction Result}\label{sec:dimension-reduction}

Under the assumption that a branched immersion \(u:\Sigma\to B^n\subset \R^n(c)\) has non-zero parallel mean curvature in $B^n$, non-zero constant contact angle along $u(\partial \Sigma) \subset \mathcal{K}$ and its image lies in a totally umbilical hypersurface \(M^{n-1}\), in this section, we show that the PMC condition descends to \(M\), i.e. the mean curvature \(H\) in $M$ is parallel in the normal bundle of \(u(\Sigma)\subset M\). Moreover, non-zero contact angle condition excludes the degenerate case \(\langle \boldsymbol{n},{\boldsymbol{\xi}}\rangle=\pm1\) and the constant contact angle condition is preserved after intersecting with \(M\), that is, \(u(\partial\Sigma)\) meets \(C=M\cap\partial B^n\) inside \(M\) at a non-zero constant angle \(\theta_M\) given by
\begin{equation*}
    \sin\theta_M=\frac{\sin\theta}{\sqrt{1-\langle \boldsymbol{n},{\boldsymbol{\xi}}\rangle^2}}\neq 0,
\end{equation*}
where $\boldsymbol{\xi}$ is the unit normal field of $M$.

\begin{prop}\label{lem:PMC-and-angle-in-umbilical-S}
Let $(N,h)=\R^n(c)$ be an $n$-dimensional space form of constant sectional curvature $c \in \R$ and let $B^n\subset N$ be a geodesic ball with smooth boundary $\partial B^n$ and outward unit normal vector $\boldsymbol{n}$.
Let $u: \Sigma \rightarrow B^n$ be a branched immersed surface with free boundary constraint $u(\partial\Sigma) \subset \partial B^n$. Assume: 
\begin{enumerate}
  \item[(i)] the mean curvature vector $\bar H$ of $\Sigma\subset N$ is non-zero and parallel in the normal bundle, i.e.\ $\nabla^\perp \bar H\equiv 0$ and $\bar H\not\equiv 0$;
  \item[(ii)] $u(\Sigma)$ meets $\partial B^n$ at a constant non-zero contact angle $\theta\in(0,\pi/2]$ along $\partial\Sigma$ satisfying
  \begin{equation}\label{eq:capillary-angle}
  \langle \boldsymbol{\nu},\boldsymbol{n} \rangle = \sin\theta \neq 0 \quad \text{on } u(\partial\Sigma),
  \end{equation}
  where $\boldsymbol{\nu}$ is the outward unit conormal of $u(\partial\Sigma)\subset u(\Sigma)$;
  \item[(iii)] $u(\Sigma)$ is contained in a totally umbilical hypersurface $M^{n-1}\subset N$ with unit normal field ${\boldsymbol{\xi}}$.
\end{enumerate}
Then, the following holds
\begin{enumerate}
  \item[(a)] The mean curvature vector $H$ of $\Sigma$ viewed as a surface in $M$ is parallel in the normal bundle of $\Sigma\subset M$, i.e.\ $\nabla^{\perp,M}H\equiv 0$.
  \item[(b)] Along $\partial\Sigma$, $u(\partial\Sigma)$ has a non-zero constant contact angle $\theta_M$ with the smooth totally umbilical hypersurface $C:=M\cap \partial B^n$ inside $M$, and
  \begin{equation}\label{eq:thetaS-formula}
    \sin\theta_M \;=\; \frac{\sin\theta}{\sqrt{1-\langle \boldsymbol{n},{\boldsymbol{\xi}}\rangle^2}}
    \;=\; \frac{\sin\theta}{\sin\alpha},
  \end{equation}
  where $\alpha\in(0,\pi)$ is the (constant) contact angle between the hypersurfaces $M$ and $\partial B^n$ along $\partial\Sigma$, defined by $\cos\alpha:=\langle \boldsymbol{n},{\boldsymbol{\xi}}\rangle$.
  In particular, the degenerate case $\langle \boldsymbol{n},{\boldsymbol{\xi}}\rangle=\pm 1$ cannot occur along $\partial\Sigma$, hence $\theta_M$ is well-defined given by \eqref{eq:thetaS-formula}.
\end{enumerate}
\end{prop}

\begin{proof}
To distinguish, during the proof of Proposition \ref{lem:PMC-and-angle-in-umbilical-S}, let $\overline{\nabla}$ be the Levi-Civita connection of $N$, $\nabla^M$ that of $M$, and $\nabla^\Sigma$ that of $\Sigma$. Let $\bar A$ be the second fundamental form of $u(\Sigma)\subset N$, $A$ that of $u(\Sigma)    \subset M$, and $A^M$ that of $M\subset N$.
Since $M$ is totally umbilical, there exists a function $\lambda$ on $M$ such that
\[
A^M(X,Y)=\lambda\,\langle X,Y\rangle\,{\boldsymbol{\xi}} \qquad \text{for all }X,Y\in TM.
\]
For $X,Y\in du(T\Sigma)\subset TM$, comparing the Gauss formulas for $u(\Sigma) \subset N$ and the composition $u(\Sigma)\subset M\subset N$ gives
\begin{equation}\label{eq:B-decomp}
\bar A(X,Y)=A(X,Y)+A^M(X,Y)=A(X,Y)+\lambda\,\langle X,Y\rangle\,{\boldsymbol{\xi}}.
\end{equation}
Taking the trace over an orthonormal frame $\{e_1,e_2\}$ of $T\Sigma$ yields the mean curvature decomposition
\begin{equation}\label{eq:H-decomp}
\bar H = H + \lambda\,{\boldsymbol{\xi}}.
\end{equation}

We next compare normal connections. Along $\Sigma$, the ambient normal bundle splits orthogonally as
\[
N_\Sigma^{N} = N_\Sigma^{M}\oplus \mathrm{span}\{{\boldsymbol{\xi}}\},
\]
where $N_\Sigma^{M} := (T\Sigma)^\perp\cap TM$ is the normal bundle of $\Sigma$ inside $M$.
If ${\boldsymbol{\eta}}_M$ is a section of $N_\Sigma^{M}$ and $X\in T\Sigma$, then ${\boldsymbol{\eta}}_M\in TM$ and $A^M(X,{\boldsymbol{\eta}}_M)=\lambda\langle X,{\boldsymbol{\eta}}_M\rangle{\boldsymbol{\xi}}=0$, hence
\[
\overline{\nabla}_X {\boldsymbol{\eta}}_M = \nabla^M_X {\boldsymbol{\eta}}_M,
\]
so taking normal components gives
\begin{equation}\label{eq:norm-conn-agree}
\nabla^\perp_X {\boldsymbol{\eta}}_M = \nabla^{\perp,M}_X {\boldsymbol{\eta}}_M.
\end{equation}
Moreover, the Weingarten formula for $M\subset N$ gives $\overline{\nabla}_X{\boldsymbol{\xi}}=-A_{\boldsymbol{\xi}} X$. Since $M$ is umbilical,
\begin{equation}\label{eq:d-xi}
\overline{\nabla}_X{\boldsymbol{\xi}}=-\lambda X \in T\Sigma,
\qquad\Rightarrow\qquad
\nabla^\perp_X{\boldsymbol{\xi}}=0.
\end{equation}
Now differentiate \eqref{eq:H-decomp} in the ambient normal bundle to get
\[
0=\nabla^\perp_X \bar H
=\nabla^\perp_X H + \nabla^\perp_X(\lambda{\boldsymbol{\xi}})
=\nabla^{\perp,M}_X H + X(\lambda)\,{\boldsymbol{\xi}} + \lambda\,\nabla^\perp_X{\boldsymbol{\xi}},
\]
using \eqref{eq:norm-conn-agree}.
By \eqref{eq:d-xi}, $\nabla^\perp_X{\boldsymbol{\xi}}=0$.
Finally, since $N$ is a space form, the Codazzi equation \eqref{eq:Codazzi} for the totally umbilical hypersurface $M$ implies $\lambda$ is constant on $M$, hence $X(\lambda)=0$ for all $X\in TM$.
Therefore $\nabla^{\perp,M}_XH=0$ for all $X\in T\Sigma$, i.e.\ $\nabla^{\perp,M}H\equiv 0$, proving (a).

Before proving (b), we first show the transversality of $M$ and $\partial B^n$ to exclude the degenerate case $\langle \boldsymbol{n},{\boldsymbol{\xi}}\rangle=\pm 1$.
Recall that $\boldsymbol{\tau}\in T(\partial B^n)\cap T\Sigma\subset TM$ and $\langle \boldsymbol{\tau},\boldsymbol{\nu}\rangle=0$.
Assume for contradiction that at some $p\in\partial\Sigma$ one has $\langle \boldsymbol{n}(p),{\boldsymbol{\xi}}(p)\rangle=\pm 1$, i.e.\ $\boldsymbol{n}(p) = \pm {\boldsymbol{\xi}}(p)$.
Then
\[
T_pM = {\boldsymbol{\xi}}(p)^\perp = \boldsymbol{n}(p)^\perp = T_p(\partial B^n),
\]
and since $du(T_p\Sigma)\subset T_pM$, we obtain $du(T_p\Sigma)\subset T_p(\partial B^n)$, in particular $\langle \boldsymbol{\nu}(p),\boldsymbol{n}(p)\rangle=0$. This forces $\sin\theta=0$, contradicting the non-zero angle assumption \eqref{eq:capillary-angle}.
Thus
\begin{equation}\label{eq:transverse}
|\langle \boldsymbol{n},{\boldsymbol{\xi}}\rangle|<1 \quad \text{everywhere on }\partial\Sigma,
\end{equation}
so the intersection $C:=M\cap \partial B^n$ is a smooth hypersurface in $M$.

Along $\partial\Sigma$, the unit normal $\boldsymbol{c}$ of $C\subset M$ can be given by projecting $\boldsymbol{n}$ onto $TM$ and normalizing:
\begin{equation}\label{eq:mu-def}
\boldsymbol{c} := \frac{\boldsymbol{n} - \langle \boldsymbol{n},{\boldsymbol{\xi}}\rangle\,{\boldsymbol{\xi}}}{\sqrt{1-\langle \boldsymbol{n},{\boldsymbol{\xi}}\rangle^2}}
\in TM,
\end{equation}
which is well-defined by \eqref{eq:transverse}.
Hence, the contact angle $\theta_M$ between $\Sigma$ and $C$ inside $M$ satisfies
\begin{equation}\label{eq:capillary-angle-S}
\langle \boldsymbol{\nu},\boldsymbol{c}\rangle = \sin\theta_M \quad \text{on }\partial\Sigma.
\end{equation}
Since $\boldsymbol{\nu}\in du(T\Sigma)\subset TM$, we have $\langle \boldsymbol{\nu}, {\boldsymbol{\xi}}\rangle=0$, and therefore by \eqref{eq:mu-def}
\[
\langle \boldsymbol{\nu},\boldsymbol{c}\rangle
=\frac{\langle \boldsymbol{\nu},\boldsymbol{n}\rangle - \langle \boldsymbol{n},{\boldsymbol{\xi}}\rangle\langle \boldsymbol{\nu},{\boldsymbol{\xi}}\rangle}{\sqrt{1-\langle \boldsymbol{n},{\boldsymbol{\xi}}\rangle^2}}
=\frac{\langle \boldsymbol{\nu},\boldsymbol{n}\rangle}{\sqrt{1-\langle \boldsymbol{n},{\boldsymbol{\xi}}\rangle^2}}
=\frac{\sin\theta}{\sqrt{1-\langle \boldsymbol{n},{\boldsymbol{\xi}}\rangle^2}}.
\]
This proves the first identity in \eqref{eq:thetaS-formula}.

It remains to show that $\theta_M$ is constant along $\partial\Sigma$.
Set $f:=\langle \boldsymbol{n},{\boldsymbol{\xi}}\rangle$ along $\partial\Sigma$. Differentiating along the unit tangent $\boldsymbol{\tau}$ gives
\[
\nabla_{\boldsymbol{\tau}} f=\langle \overline{\nabla}_{\boldsymbol{\tau}}\boldsymbol{n},{\boldsymbol{\xi}}\rangle+\langle \boldsymbol{n},\overline{\nabla}_{\boldsymbol{\tau}}{\boldsymbol{\xi}}\rangle.
\]
Because $\partial B^n$ is a geodesic sphere in a space form, it is totally umbilical, so $\overline{\nabla}_{\boldsymbol{\tau}}\nu=-\kappa\,\boldsymbol{\tau}$ for some constant $\kappa$. And $M$ is totally umbilical, $\overline{\nabla}_{\boldsymbol{\tau}}{\boldsymbol{\xi}}=-\lambda\,\boldsymbol{\tau}$.
Using $\langle \boldsymbol{\tau},{\boldsymbol{\xi}}\rangle=0$ and $\langle \boldsymbol{\tau},\nu\rangle=0$, we obtain $\nabla_{\boldsymbol{\tau}}f=0$.
Hence, $f=\langle \boldsymbol{n},{\boldsymbol{\xi}}\rangle$ is constant along each connected component of $u(\partial\Sigma)$, and $\theta_M$ is constant by the constancy of $\theta$.
Writing $f=\cos\alpha$  gives the second identity in \eqref{eq:thetaS-formula}.
This proves (b) and completes the proof of Proposition \ref{lem:PMC-and-angle-in-umbilical-S}.
\end{proof}

\vspace{1cm}
\section{Proof of Main Theorem}\label{sec:proof-main}
In this section, we complete the proofs of the main results---Theorem~\ref{main theorem 1} and Theorem~\ref{main theorem 2}.

\subsection{Proof of Theorem~\ref{main theorem 1}}\label{sec:proof-main 4.1}\ 
\vskip5pt


We begin with the proof of Theorem \ref{main theorem 1}, where the  proof is inspired by the holomorphic quadratic differential method utilized in \cite{yau1974,hoffman1973jdg,chen-Chern72,Fraser-Schoen2015}, combined with our dimensional reduction Proposition \ref{lem:PMC-and-angle-in-umbilical-S}.
\begin{proof}[\textbf{Proof of Theorem \ref{main theorem 1}}]
Let $u : D \rightarrow B^4$ be a branched immersion with non-zero parallel mean curvature vector from the unit disk $D$ into a geodesic ball $B^4$ of a $4$-dimensional space $\R^4(c)$ such that $u(D)$ meets $\partial B^4$ at a constant angle $\theta \neq 0$. Work on $u(D)$ away from branch points.  Combining \eqref{eq:defi holomorphic 1} with Lemma \ref{lem:holomorphic}, $\Phi_H$ and $\Phi_{H^\prime}$ are holomorphic on $D$ away from branch points, and they extend holomorphically across branch points.

First, we claim the following:
\smallskip
\claim\label{claim 1}  Either $\Phi_H\equiv 0$, or else $\beta:=\Phi_{H^\prime}/\Phi_H$ is a real constant. In either case there exists a parallel unit
normal field ${\boldsymbol{\eta}}$ and a constant $\lambda$ such that $A_{{\boldsymbol{\eta}}}=\lambda I$.

\begin{proof}[\textbf{Proof of Claim~\ref{claim 1}}]
Because $H,H^\prime$ are parallel, the normal bundle is flat by Lemma \ref{lem:properties H} and the Ricci equation \eqref{eq:Ricci} in a space form gives
$[A_{H},A_{H^\prime}]=0$, hence also $[A_{H}^\circ,A_{H^\prime}^\circ]=0$, where $A^\circ_{{\boldsymbol{\xi}}}$ denotes the trace-free part of the shape operator $A_{\boldsymbol{\xi}}$ in direction ${\boldsymbol{\xi}}$. Recall that in dimension two, traceless symmetric operators have a unique pair of principal directions unless zero, so commuting forces them to share those directions, hence one is a scalar multiple of the other. Then, away from branch points, a direct computation yields
\[
[A_{H}^\circ,A_{H^\prime}^\circ]= \frac{8}{\abs{\nabla u}^4} \Im(\varphi(H^\prime)\overline{\varphi(H)}) \mathcal{J},
\]
where $\mathcal{J}$ is the complex structure on $u(D)$.
So $[A_{H}^\circ,A_{H^\prime}^\circ]=0$ implies $\Im(\varphi(H^\prime)\overline{\varphi(H)})=0$.
Therefore on the open set $\{\Phi_H\neq 0\}$ the quotient $\beta=\varphi(H^\prime)/\varphi(H)=\Phi_{H^\prime}/\Phi_H$ is real--valued.
Since $\beta$ is meromorphic and real--valued on a dense open set, the open mapping theorem forces $\beta$ to be constant unless $\Phi_H\equiv 0$.

If $\Phi_H\equiv 0$, set ${\boldsymbol{\eta}}=H$; then $A_{{\boldsymbol{\eta}}}^\circ\equiv 0$ and $A_{{\boldsymbol{\eta}}}=\lambda I$ with
$\lambda=\tfrac12\mathrm{trace} A_{H}=|H|$ being constant.
If $\Phi_H\not\equiv 0$, write $\beta=-\tan\alpha$ with constant $\alpha$ and define
\[
{\boldsymbol{\eta}} := (\sin\alpha)\,H+(\cos\alpha)\,H^\prime.
\]
Then ${\boldsymbol{\eta}}$ is parallel and
$\Phi_{\boldsymbol{\eta}}=(\sin\alpha)\Phi_H+(\cos\alpha)\Phi_{H^\prime}=(\sin\alpha+\beta\cos\alpha)\Phi_H\equiv 0$,
so $A_{{\boldsymbol{\eta}}}^\circ\equiv 0$ and $A_{{\boldsymbol{\eta}}}=\lambda I$ with constant $\lambda=\tfrac12\mathrm{trace} A_{{\boldsymbol{\eta}}}$. Whenever in any cases, we prove the Claim \ref{claim 1}.
\end{proof}

By Claim~\ref{claim 1}, there is a parallel unit normal field ${\boldsymbol{\eta}}$ and a constant $\lambda$ such that $A_{{\boldsymbol{\eta}}}=\lambda I$.
We show that this forces $u(D)$ to lie in a $3$--dimensional totally umbilic hypersurface $P^3$ with constant principal curvature
$\lambda$.

\noindent\textbf{Step 1}: $u(D)$ is contained in a $3$-dimensional totally umbilic hypersurface $P^3$ of $\R^4(c)$.
\smallskip

\noindent\emph{Case $c=0$ (Euclidean).}
In $\R^4$ the Weingarten formula gives, for any tangent $X$ to $u(D)$,
\[
\overline{\nabla}_X {\boldsymbol{\eta}} = -A_{{\boldsymbol{\eta}}}X + \nabla^\perp_X{\boldsymbol{\eta}} = -\lambda X.
\]
If $\lambda=0$, then $\overline{\nabla}{\boldsymbol{\eta}}\equiv 0$, so ${\boldsymbol{\eta}}$ is a constant vector and $u(D)$ lies in an affine $3$--plane.
If $\lambda\neq 0$, define $p:=x+\lambda^{-1}{\boldsymbol{\eta}}$ where $x$ is the position vector of $u(D)$ in $\R^4$. Then,
$\overline{\nabla}_X p = X + \lambda^{-1}\overline{\nabla}_X{\boldsymbol{\eta}} =0$, so $p$ is constant and
$|x-p|=|\lambda|^{-1}$; hence $u(D)$ lies in a round $3$-sphere $\S^3(p,|\lambda|^{-1})$.
In either case, $u(D)\subset P^3$ with $P^3$ being totally umbilic in $\R^4$ and having principal curvature $\lambda$.

\smallskip
\noindent\emph{Case $c>0$ (sphere).}
Realize $\S^4(c)$ as $\{x\in\R^5:\langle x,x\rangle=1/c\}$.
Since ${\boldsymbol{\eta}}$ is tangent to $\S^4(c)$ along $u(D)$, one has $\langle x,{\boldsymbol{\eta}}\rangle=0$ and $\nabla^{\R^5}_X {\boldsymbol{\eta}} = \overline{\nabla}_X {\boldsymbol{\eta}}$.
Thus, for tangent $X$ to $u(D)$, $\overline{\nabla}_X{\boldsymbol{\eta}}=-\lambda X$ as above. Define the ambient vector
\[
a:={\boldsymbol{\eta}}+\lambda x\in\R^5.
\]
Then $\nabla^{\R^5}_X a = \overline{\nabla}_X{\boldsymbol{\eta}} + \lambda\nabla^{\R^5}_X x =0$, so $a$ is constant.
Consequently $\langle x,a\rangle=\lambda\langle x,x\rangle=\lambda/c$ is constant on $u(D)$, hence
\[
u(D) \subset P^3 := \{y\in\S^4(c):\langle y,a\rangle=\lambda/c\},
\]
a projective hyperplane section of $\S^4(c)$, i.e. a round 3-sphere.
Moreover, $P^3\subset \S^4(c)$ is totally umbilic with shape operator $\lambda I$ with sectional curvature $c_1=c+\lambda^2\ge c$.

\smallskip
\noindent\emph{Case $c<0$ (hyperbolic).}
Write $c=-k^2$ ($k>0$) and use the hyperboloid model
$\mathbb{H}^4(c)=\{x\in\R^{4,1}:\langle x,x\rangle_1=1/c=-1/k^2\}$ in Minkowski space $\R^{4,1}$.
Exactly as in the spherical case, from $\overline{\nabla}_X{\boldsymbol{\eta}}=-\lambda X$ define
\[
a:={\boldsymbol{\eta}}+\lambda x\in\R^{4,1},
\]
which is constant, and obtain
\[
u(D) \subset P^3 := \{y\in\mathbb{H}^4(c):\langle y,a\rangle_1=\lambda/c\}.
\]
Again $P^3$ is totally umbilic with principal curvature $\lambda$.
Its sectional curvature is $c_1=c+\lambda^2$. The different choices  of $\lambda$
yields the standard classification:
\[
\lambda=0 \Rightarrow P^3 \text{ totally geodesic},\quad
0<|\lambda|<k \Rightarrow P^3 \text{ equidistant},
\]
\begin{equation*}
    |\lambda|=k \Rightarrow P^3 \text{ horosphere},\quad
|\lambda|>k \Rightarrow P^3 \text{ geodesic sphere}.
\end{equation*}
\noindent\textbf{Step 2}: $u(D)$ is totally umbilic in $P^3$.

Fix the $P^3$ found above. By Proposition~\ref{lem:PMC-and-angle-in-umbilical-S},  $u(D)$ is a constant mean curvature surface (CMC) in the $3$-dimensional space form $P^3$ with non-zero constant contact angle along $u(\partial D)\subset \partial B^4 \cap P^3$.

Let $A^P$ be the second fundamental form of $u(D)\subset P^3$ and define the Hopf type quartic differential
\[
\mathcal Q := \inner{A^P\left(\frac{\partial u}{\partial z},\frac{\partial u}{\partial z}\right), A^P\left(\frac{\partial u}{\partial z},\frac{\partial u}{\partial z}\right)}\,dz^4 = \phi(z) dz^4.
\]
For a CMC surface in a $3$-space form, Codazzi equation implies $\mathcal Q$ is holomorphic on $u(D)$ away from branch points, and it extends holomorphically across branch points.

\smallskip
We now show $ z^4 \mathcal Q$ is real along $\partial D$. Set $S_P:=\partial B^4\cap P^3$, which is totally umbilic in $P^3$ by Proposition \ref{lem:PMC-and-angle-in-umbilical-S}, i.e. $\nabla^{P}_{\boldsymbol{\tau}} \boldsymbol{n}_P = -\kappa\,\boldsymbol{\tau}$ for some constant $\kappa$. Here, $\boldsymbol{n}_P$ is the outward unit normal of $S_P\subset P^3$, which coincides with $\boldsymbol{\mu}$ as in Definition \ref{defi:contact angle}.

Differentiate the constant angle condition $\langle \boldsymbol{\nu},\boldsymbol{n}_P\rangle=\cos\theta_P$ along $\boldsymbol{\tau}$ to get
\[
0=\boldsymbol{\tau}\langle \boldsymbol{\nu},\boldsymbol{n}_P\rangle = \langle \nabla^{P}_{\boldsymbol{\tau}}\boldsymbol{\nu},\boldsymbol{n}_P\rangle + \langle \boldsymbol{\nu},\nabla^{P}_{\boldsymbol{\tau}}\boldsymbol{n}_P\rangle
= \langle \nabla^{P}_{\boldsymbol{\tau}}\boldsymbol{\nu},\boldsymbol{n}_P\rangle.
\]
Since $\boldsymbol{\tau}\perp \boldsymbol{\nu}$ and ${\boldsymbol{\tau}}\perp \boldsymbol{n}_P$, the vectors $\boldsymbol{\nu}$ and the unit normal $N_P$ of $du(TD)$ in $TP$ span the orthogonal complement of ${\boldsymbol{\tau}}$ in $T(P^3)$, and
\[
\boldsymbol{n}_P = (\cos\theta_P)\,\boldsymbol{\nu} + (\sin\theta_P)\,N_P.
\]
Therefore
\[
0=\langle \nabla^{P}_{\boldsymbol{\tau}}\boldsymbol{\nu},\boldsymbol{n}_P\rangle
=(\sin\theta_P)\,\langle \nabla^{P}_{\boldsymbol{\tau}}\boldsymbol{\nu},N_P\rangle
=(\sin\theta_P)\,A^P({\boldsymbol{\tau}},\boldsymbol{\nu}).
\]
Since $\sin\theta_P\neq 0$, we obtain $A^P({\boldsymbol{\tau}},\boldsymbol{\nu})=0$ along $u(\partial D)$.

For the convenience, we consider the polar coordinate $(\rho, \varphi)$ on $D$ and take the transformation $w = \log \rho + \sqrt{-1} \varphi$. Then, we have 
     \begin{equation*}
         A^P\p{\frac{\partial u}{\partial z}, \frac{\partial u}{\partial z}} = \frac{1}{z^2}A^P\p{\frac{\partial u}{\partial w}, \frac{\partial u}{\partial w}}
     \end{equation*}
     and on $\partial D$ there holds
     \begin{equation*}
         \frac{\partial}{\partial w} = \frac{1}{2}\p{\frac{\partial }{\partial \rho} + \sqrt{-1} \frac{\partial}{\partial \varphi}}. 
     \end{equation*}
     Therefore, on $u(\partial D)$,  the imaginary part of $z^2 A^P(\partial_z u, \partial_z u)$ can be rewritten as
     \begin{align} \label{eq1.5:proof theorem 1}
            \Im\p{ z^2 A^P\p{\frac{\partial u}{\partial w}, \frac{\partial u}{\partial w}}} &= - \abs{\nabla u}^2 A^P(\boldsymbol{\tau}, \boldsymbol{\nu})= 0,
    \end{align}
which in particular means $z^4 \mathcal{Q}(z)$ is real on $\partial D$. Since $z^4 \mathcal Q$ is holomorphic on $D$ and real on $\partial D$, Schwarz reflection tells us $z^4\mathcal Q\equiv c$ on $D$ for some $c \in \R_+$ by \eqref{eq1.5:proof theorem 1}. Since $z^4 \mathcal{Q}(z)$ vanishes at the origin, 
\begin{align*}
     z^4 \mathcal{Q}(z) &= \frac{1}{16}\abs{\rho^2 A^P\p{\frac{\partial u}{\partial \rho}, \frac{\partial u}{\partial\rho}}- A^P\p{\frac{\partial u}{\partial \varphi}, \frac{\partial u}{\partial\varphi}}}^2 - \frac{1}{4}\abs{\rho\, A^P\p{\frac{\partial u}{\partial \rho}, \frac{\partial u}{\partial \varphi}}}^2\nonumber\\
     &\quad + \frac{\sqrt{-1}}{4} \inner{\rho^2 A^P\p{\frac{\partial u}{\partial \rho}, \frac{\partial u}{\partial\rho}}- A^P\p{\frac{\partial u}{\partial \varphi}, \frac{\partial u}{\partial\varphi}}, \rho\, A^P\p{\frac{\partial u}{\partial \rho}, \frac{\partial u}{\partial \varphi}}}\nonumber\\
     &\equiv 0.
\end{align*}
This means, the holomorphic section $z^2 \mathcal{Q}(z)$ vanishes on $\partial D$, hence it vanishes on whole $D$.  Therefore, the trace-free part of the shape operator of $u(D)\subset P^3$ vanishes and $u(D)$ is totally umbilic in $P^3$. The remaining assertion of Theorem \ref{main theorem 1} follows from the classical classification of totally umbilic surfaces in three dimensional space forms, see for instance \cite[Chapter 7]{Spivak-vol4}.

Replacing the domain \( D \) by an annulus \( A(r,R) \) with \( 0<r<R \), an entirely analogous argument shows that the associated holomorphic quadratic differential satisfies
\[
z^{4}\mathcal{Q}(z)\equiv c
\quad \text{on } A(r,R),
\]
for some constant \( c\in\mathbb{R}_{+} \). If \( c=0 \), then the branched immersion \( u(A(r,R)) \) is totally umbilic and is therefore completely classified by Theorem~\ref{main theorem 1}. However, in all such cases, no totally umbilic surface admits a branched immersed annulus piece meeting the boundary with a non-zero constant contact angle. Consequently, the case \( c=0 \) is excluded. It follows that \( c>0 \), and hence \( u(A(r,R)) \) has no umbilic points. This completes the proof of Theorem \ref{main theorem 1}.
\end{proof}

\subsection{Proof of Theorem \ref{main theorem 2}}\label{sec:proof-main 4.2}\ 
\vspace{5pt}

In this subsection, we prove Theorem \ref{main theorem 2}.  The main input is the space-form codimension reduction Lemma~\ref{lem:flat-normal-implies-R4c} below.
More precisely, from the assumptions $\nabla^\perp H=0$ and $R^\perp=0$, we construct two normal vector fields
\[
{\boldsymbol{\xi}}_3:=\frac{H}{|H|}\qquad\text{and}\qquad {\boldsymbol{\xi}}_4,
\]
which are parallel in the normal bundle over $\Sigma_0:=\Sigma\setminus \mathcal{B},$ such that
\[
A\big(du(T\Sigma_0),\,du(T\Sigma_0)\big)\subset \mathrm{span}\{{\boldsymbol{\xi}}_3,{\boldsymbol{\xi}}_4\}.
\]
Equivalently, the rank-$4$ subbundle $du(T\Sigma_0)\oplus \mathrm{span}\{{\boldsymbol{\xi}}_3,{\boldsymbol{\xi}}_4\}$
is parallel with respect to $\overline{\nabla}$ along $u(\Sigma_0)$.
Integrating this parallel subbundle in the ambient space form $\R^n(c)$ and utilizing a global gluing argument yield a totally geodesic $P^4\cong \R^4 (c)$ containing $u(\Sigma)$.

\begin{proof}[\textbf{Proof of Theorem \ref{main theorem 2}}]
Let $\mathcal B\subset \Sigma$ be the set of branch points of $u$ and set $\Sigma_0:=\Sigma\setminus \mathcal B$. All  geometric identities below are proved on $\Sigma_0$ and then extend across $\mathcal B$ by smooth extension. The proof of Theorem \ref{main theorem 2} is divided into two cases:

\smallskip
\noindent\textbf{Case (I)}: $\Phi_H\equiv 0$ on $\Sigma_0$. 

In this case, $A_{\boldsymbol{\xi}} = |H| I$ on $du(T\Sigma_0)$. By a similar argument as Step 1 in proof of Theorem~\ref{main theorem 1}, there exists a totally umbilic hypersurface $M^{n-1}\subset \R^n(c)$
containing $u(\Sigma_0)$, whose unit normal along $u(\Sigma_0)$ equals $H/|H|$ and whose umbilicity factor equals $|H|$.

Now view $\Sigma_0\subset M^{n-1}\subset \R^n(c)$. For $X,Y\in du(T\Sigma_0)$,
\[
A^{\R^n(c)}_{\Sigma}(X,Y) = A^{M^{n-1}}_{\Sigma}(X,Y) + A^{\R^n(c)}_{M^{n-1}}(X,Y)
= A^{M^{n-1}}_{\Sigma}(X,Y) + \langle X,Y\rangle\,H.
\]
Taking traces gives the mean curvature vector of $\Sigma_0$ in $M^{n-1}$ vanishes on $\Sigma_0$. Hence $\Sigma_0$ is minimal in $M^{n-1}$, which yields the first alternative of Theorem \ref{main theorem 1} after taking closures to include branch points.

\smallskip
\noindent\textbf{Case (II)}: $\Phi_H\not\equiv 0$ on $\Sigma_0$.

Since $\Phi_H$ is holomorphic, the zero set of $\Phi_H$ is discrete, that is, the set
\[
\Omega:=\{p\in \Sigma_0:\ \Phi_H(p)\neq 0\}
\]
is open and dense.

Fix $p\in\Omega$. Choose an orthonormal local normal frame $\{{\boldsymbol{\xi}}_3,\dots,{\boldsymbol{\xi}}_n\}$ around $p$ so that
\[
{\boldsymbol{\xi}}_3=H/|H|.
\]
Since $\nabla^\perp H=0$, we have $\nabla^\perp {\boldsymbol{\xi}}_3=0$, in terms of normal connection 1-forms $\omega_{\alpha\beta}$ defined by
$\nabla^\perp {\boldsymbol{\xi}}_\alpha = \sum_{\beta}\omega_{\alpha\beta}\,{\boldsymbol{\xi}}_\beta$, this is equivalent to $\omega_{3\beta}\equiv 0$ for $\beta \ge 4$ on $\Sigma_0$. Applying the Ricci equation \eqref{eq:Ricci} with ${\boldsymbol{\xi}}={\boldsymbol{\xi}}_3$ and ${\boldsymbol{\eta}}={\boldsymbol{\xi}}_\alpha$ and using $\omega_{3\beta}=0$ gives
\begin{equation}\label{eq:commute-3r}
[A_{{\boldsymbol{\xi}}_3},A_{{\boldsymbol{\xi}}_\alpha}]=0,\qquad \alpha = 4,\dots,n.
\end{equation}
Also, since ${\boldsymbol{\xi}}_\alpha\perp H$, we have
\begin{equation}\label{eq:trace0}
\mathrm{tr}(A_{{\boldsymbol{\xi}}_\alpha})=2\langle H,{\boldsymbol{\xi}}_\alpha\rangle=0,\qquad \alpha = 4,\dots,n.
\end{equation}
At $p\in\Omega$, $\Phi_H(p)\neq 0$ means $A_{{\boldsymbol{\xi}}_3}$ is not a multiple of the identity at $p$.
For $2\times 2$ symmetric operators, \eqref{eq:commute-3r} implies that each $A_{{\boldsymbol{\xi}}_\alpha}(p)$ is diagonalizable
in the same orthonormal basis that diagonalizes $A_{{\boldsymbol{\xi}}_3}(p)$. And by \eqref{eq:trace0},  any two $A_{{\boldsymbol{\xi}}_\alpha}(p),A_{{\boldsymbol{\xi}}_\beta}(p)$ commute for $\alpha, \,\beta\geq 3$, so
\[
[A_{{\boldsymbol{\xi}}},A_{{\boldsymbol{\eta}}}](p)=0 \quad \text{for all normal vectors }{\boldsymbol{\xi}},{\boldsymbol{\eta}}\text{ at }p.
\]
Therefore $R^\perp(p)=0$ for every $p\in\Omega$, and by continuity $R^\perp\equiv 0$ on all of $\Sigma_0$.

\begin{lemma}
\label{lem:flat-normal-implies-R4c}
Let $u:\Sigma\to \R^n(c)$ be a branched immersion from a connected bordered Riemann surface $\Sigma$, and set
$\Sigma_0:=\Sigma\setminus\mathcal B$ where $\mathcal B$ is the branch set.
Assume on $\Sigma_0$ that
\[
\nabla^\perp H\equiv 0,\qquad H\not\equiv 0,\qquad R^\perp\equiv 0
\]
Then there exists a totally geodesic submanifold
$P^4\cong \R^4(c)\subset \R^n(c)$ such that $u(\Sigma_0)\subset P^4$.
\end{lemma}

\begin{proof}

Fix $p\in\Sigma_0$.
Let $ d u(T_p\Sigma_0)^\perp$ be the normal space and set
\[
{\boldsymbol{\xi}}_3(p):=\frac{H(p)}{|H(p)|},\qquad
N_p^0:=\{{\boldsymbol{\eta}}\in d u(T_p\Sigma_0)^\perp:\langle {\boldsymbol{\eta}},{\boldsymbol{\xi}}_3(p)\rangle=0\}.
\]
For ${\boldsymbol{\eta}}\in N_p^0$ we have $\mathrm{trace}(A_{\boldsymbol{\eta}})=2\langle H,{\boldsymbol{\eta}}\rangle=0$, so $A_{\boldsymbol{\eta}}\in\mathrm{Sym}_0(T_p\Sigma_0)$
(traceless symmetric endomorphisms).
Define the linear map
\[
\gamma_p:N_p^0\longrightarrow \mathrm{Sym}_0(T_p\Sigma_0),\qquad \gamma_p({\boldsymbol{\eta}}):=A_{\boldsymbol{\eta}}.
\]
By Ricci equation \eqref{eq:Ricci}, the assumption $R^\perp\equiv 0$ implies, for ${\boldsymbol{\eta}}_1,{\boldsymbol{\eta}}_2\in N_p^0$, the two traceless symmetric $2\times 2$ matrices
$A_{{\boldsymbol{\eta}}_1},A_{{\boldsymbol{\eta}}_2}$ commute, hence they are simultaneously diagonalizable. Since they are traceless,
each is a multiple of $\mathrm{diag}(1,-1)$ in that basis, so they are linearly dependent.
Thus
\[
\dim \mathrm{Im}(\gamma_p)\le 1.
\]
Let $\mathcal O_p:=\ker(\gamma_p)=\{{\boldsymbol{\eta}}\in N_p^0:A_{\boldsymbol{\eta}}=0\}$ and define
\[
N'_p:=\mathcal O_p^\perp\cap N_p^0.
\]
Then $\dim N'_p=\dim\mathrm{Im}(\gamma_p)\le 1$.

Define
\[
M_1:=\{p\in\Sigma_0:\dim N'_p=1\},\qquad M_0:=\{p\in\Sigma_0:\dim N'_p=0\}.
\]
Then $M_1$ is open. Let $U\subset M_1$ be a  connected open local chart. Since $\dim N'_p=1$ on $U$, we can choose a smooth unit normal field
${\boldsymbol{\xi}}_4$ on $U$ spanning $N'_p$.
Extend $\{{\boldsymbol{\xi}}_3,{\boldsymbol{\xi}}_4\}$ to a local orthonormal normal frame
$\{{\boldsymbol{\xi}}_3,{\boldsymbol{\xi}}_4,{\boldsymbol{\xi}}_5,\dots,{\boldsymbol{\xi}}_n\}$ on $U$ such that
\[
{\boldsymbol{\xi}}_\alpha(p)\in \mathcal O_p, \quad\text{for all }p\in U,\ \alpha \ge 5.
\]
Equivalently,
\begin{equation}\label{eq:Axi_r_zero}
A_{{\boldsymbol{\xi}}_\alpha}\equiv 0 \quad\text{on }U,\qquad \alpha=5,\dots,n.
\end{equation}

Because $\nabla^\perp H=0$ and $|H|\neq 0$ is constant on $U$, we have
\[
\nabla^\perp {\boldsymbol{\xi}}_3=0.
\]
This implies
\begin{equation}\label{eq:omega3r_zero}
\omega_{3\beta}\equiv 0\quad\text{on }U,\qquad \beta = 4,\dots,n.
\end{equation}

Next we prove ${\boldsymbol{\xi}}_4$ is also parallel in the normal bundle on $U$.
Use the operator form of Codazzi equation \eqref{eq:Codazzi} in the moving normal frame $\{{\boldsymbol{\xi}}_3,{\boldsymbol{\xi}}_4,{\boldsymbol{\xi}}_5,\dots,{\boldsymbol{\xi}}_n\}$,
for each $\alpha$ and all tangent vector fields $X,Y$ to $u(\Sigma_0)$, we have
\begin{equation}\label{eq:Codazzi-operator}
(\nabla_X A_{{\boldsymbol{\xi}}_\alpha})(Y)-(\nabla_Y A_{{\boldsymbol{\xi}}_\alpha})(X)
=\sum_{\beta=3}^n\Big(\omega_{\alpha\beta}(X)A_{{\boldsymbol{\xi}}_\beta}(Y)-\omega_{\alpha\beta}(Y)A_{{\boldsymbol{\xi}}_\beta}(X)\Big).
\end{equation}
Apply \eqref{eq:Codazzi-operator} with $\alpha \ge 5$.
Since $A_{{\boldsymbol{\xi}}_\alpha}\equiv 0$ by \eqref{eq:Axi_r_zero}, the left-hand side of \eqref{eq:Codazzi-operator} vanishes.
On the right-hand side, the only potentially non-zero terms are $\beta=3,4$,
but $\omega_{\alpha3}=-\omega_{3\alpha}=0$ by \eqref{eq:omega3r_zero}, so we obtain
\[
0=\omega_{\alpha 4}(X)A_{{\boldsymbol{\xi}}_4}(Y)-\omega_{\alpha 4}(Y)A_{{\boldsymbol{\xi}}_4}(X)\qquad\text{on }U,\quad \alpha \ge 5.
\]
Fix $p\in U$. Since $p\in M_1$, we have $A_{{\boldsymbol{\xi}}_4}(p)\neq 0$,  it has eigenvalues $\lambda,-\lambda$ with $\lambda\neq 0$.
Choose an orthonormal basis $(e_1,e_2)$ of $T_p\Sigma_0$ diagonalizing $A_{{\boldsymbol{\xi}}_4}(p)$ so that
$A_{{\boldsymbol{\xi}}_4}(p)(e_1)=\lambda e_1$, $A_{{\boldsymbol{\xi}}_4}(p)(e_2)=-\lambda e_2$.
Plugging $(X,Y)=(e_1,e_2)$ yields
\[
0=\omega_{\alpha 4}(e_1)(-\lambda e_2)-\omega_{\alpha 4}(e_2)(\lambda e_1),
\]
so $\omega_{\alpha 4}(e_1)=\omega_{\alpha 4}(e_2)=0$, hence $\omega_{\alpha 4}(p)=0$.
Since $p$ was arbitrary, we have
\begin{equation}\label{eq:omega_r4_zero}
\omega_{\alpha 4}\equiv 0\quad\text{on }\,\,U,\qquad \alpha=5,\dots,n.
\end{equation}
Using $\omega_{43}=-\omega_{34}=0$ from \eqref{eq:omega3r_zero},
we conclude
\begin{equation}\label{eq:xi4_parallel}
\nabla^\perp {\boldsymbol{\xi}}_4=\sum_{\beta}\omega_{4\beta}{\boldsymbol{\xi}}_\beta\equiv 0\quad\text{on }U.
\end{equation}

From \eqref{eq:Axi_r_zero}, for $\alpha \ge 5$ and all tangent $X,Y$ to $u(\Sigma_0)$,
\[
0=\langle A_{{\boldsymbol{\xi}}_\alpha}(X),Y\rangle=\langle A(X,Y),{\boldsymbol{\xi}}_\alpha\rangle.
\]
Hence $A(X,Y)$ has no component in ${\boldsymbol{\xi}}_\alpha$ for $\alpha \ge 5$, i.e.
\begin{equation}\label{eq:B_in_span_34}
A(du(TU),du(TU))\subset \mathrm{span}\{{\boldsymbol{\xi}}_3,{\boldsymbol{\xi}}_4\}\quad\text{on }U.
\end{equation}
We now use the standard linear model of space form as in the proof of Theorem \ref{main theorem 1}.
Embed $\R^n(c)$ isometrically as a quadric in a flat $(n+1)$--dimensional vector space
\[
\R^n \hookrightarrow \mathbb{E}^{n+1}_\varepsilon,
\]
where $\mathbb{E}^{n+1}_\varepsilon=\R^{n+1}$ if $c\ge 0$ and $\mathbb{E}^{n+1}_\varepsilon=\R^{n,1}$ if $c<0$,
so that the position vector $x$ of $\R^n(c)$ satisfies $\langle x,x\rangle_\varepsilon=1/c$ when $c\neq 0$.
Let $D$ be the flat connection on $\mathbb{E}^{n+1}_\varepsilon$.
Then for tangent vector fields $X,Y$ on $\R^n(c)$,
the Gauss formula for the quadric gives
\begin{equation}\label{eq:quadric-Gauss}
D_X Y = \overline{\nabla}_X Y + c\,\langle X,Y\rangle\,x.
\end{equation}
Also, $D_X x = X$ for $X$ tangent to $\R^n(c)$. Fix $p_0\in U$ and define the linear subspace
\[
W_0:=\mathrm{span}\big\{u(p_0),\ du(T_{p_0}\Sigma_0),\ {\boldsymbol{\xi}}_3(p_0),\ {\boldsymbol{\xi}}_4(p_0)\big\}\subset \mathbb{E}^{n+1}_\varepsilon.
\]
We claim that for every $p\in U$,
\begin{equation}\label{eq:W_const}
\mathrm{span}\big\{u(p),\ du(T_p\Sigma_0),\ {\boldsymbol{\xi}}_3(p),\ {\boldsymbol{\xi}}_4(p)\big\}=W_0.
\end{equation}

To prove \eqref{eq:W_const}, let $X,Y$ be tangent vector fields on $u(U)$.
Using \eqref{eq:quadric-Gauss}, we get
\begin{align*}
D_XY
&= \overline{\nabla}_XY + c\,\langle X,Y\rangle\,u \\
&= \nabla_XY + A(X,Y) + c\,\langle X,Y\rangle\,u.
\end{align*}
By \eqref{eq:B_in_span_34}, $A(X,Y)\in \mathrm{span}\{{\boldsymbol{\xi}}_3,{\boldsymbol{\xi}}_4\}$, so $D_XY$ lies in the span of
$\{u,du(T\Sigma_0),{\boldsymbol{\xi}}_3,{\boldsymbol{\xi}}_4\}$. Next, for $\alpha=3,4$,
since $\langle X,{\boldsymbol{\xi}}_\alpha\rangle=0$, we have from \eqref{eq:quadric-Gauss}
$D_X{\boldsymbol{\xi}}_\alpha=\overline{\nabla}_X{\boldsymbol{\xi}}_\alpha$.
Using the Weingarten formula on $N$,
\[
\overline{\nabla}_X{\boldsymbol{\xi}}_\alpha=-A_{{\boldsymbol{\xi}}_\alpha}(X)+\nabla^\perp_X{\boldsymbol{\xi}}_\alpha.
\]
By $\nabla^\perp{\boldsymbol{\xi}}_3=0$ and \eqref{eq:xi4_parallel}, this gives
\[
D_X{\boldsymbol{\xi}}_\alpha=-A_{{\boldsymbol{\xi}}_\alpha}X\in du(T\Sigma_0).
\]
Finally, $D_X u = X\in du(T\Sigma_0)$. Therefore, the flat derivative $D_X$ preserves the subbundle
\[
p\longmapsto \mathrm{span}\{u(p),du(T_p\Sigma_0),{\boldsymbol{\xi}}_3(p),{\boldsymbol{\xi}}_4(p)\}\subset \mathbb{E}^{n+1}_\varepsilon
\]
along $U$. Since $D$ is flat, this implies the subspace is constant along any curve in $U$,
hence \eqref{eq:W_const} holds.

Define
\[
P^4 := \R^n(c)\cap W_0 \subset \R^n(c).
\]
Because $W_0$ is a linear subspace of $\mathbb{E}^{n+1}_\varepsilon$, the submanifold $P^4$
is totally geodesic, hence is a 4-dimensional space form of curvature $c$.
Moreover, by \eqref{eq:W_const}, $u(U)\subset W_0$, so $u(U)\subset P^4$.
This proves the claim on $U\subset M_1$.

On $M_0$, we have $\dim N'_p=0$, hence $A_{\boldsymbol{\eta}}=0$ for all ${\boldsymbol{\eta}}\perp H$.
Thus $A(X,Y)\subset \mathrm{span}\{{\boldsymbol{\xi}}_3\}$ and the same argument yields the containment in a totally geodesic
$P^3\cong \R^3(c)$, which is a totally geodesic submanifold of $P^4$ in the previous case.

Assume first $M_1\neq\varnothing$ and fix one connected component $U_*\subset M_1$.
Pick $p_*\in U_*$ and let ${\boldsymbol{\xi}}_4^{(*)}$ be the unit normal field on $U_*$ spanning $N'_p$. Define
\[
W_*:=\mathrm{span}\big\{u(p_*),\ du(T_{p_*}\Sigma_0),\ {\boldsymbol{\xi}}_3(p_*),\ {\boldsymbol{\xi}}_4^{(*)}(p_*)\big\}\subset \mathbb{E}^{n+1}_\varepsilon,
\qquad
P^4_*:=\R^n(c)\cap W_*.
\]
By previous argument on $U\subset M_1$, we have $u(U_*)\subset P^4_*$. Define the propagation set
\[
\mathcal S:=\Big\{q\in \Sigma_0:\ \exists\ \text{a neighborhood } U_q\ni q \text{ such that } u(U_q)\subset P^4_*\Big\}.
\]
Then $\mathcal S$ is open and nonempty.
We claim that $\mathcal S$ is also closed in $\Sigma_0$.
Let $p\in \overline{\mathcal S}$ and choose a small connected neighborhood $W\ni p$.

\medskip
\noindent\textbf{Case II(a).} $p\in \partial U_1\cap \partial V$, where $U_1$ is a connected component of $M_1$ with $U_1\subset \mathcal S$
and $V$ is a connected component of $\mathrm{int}(M_0)$.

On $V\subset M_0$, we have $\dim N'_q=0$, hence $A_{\boldsymbol{\eta}}=0$ for all ${\boldsymbol{\eta}}\perp{\boldsymbol{\xi}}_3$.
Therefore, by the same quadric model argument of space form used above,
there exists a totally geodesic $P^3_V:=\R^n(c)\cap W_V$ such that $u(V)\subset P^3_V$ for some 4-dimensional linear space $W_V \subset \mathbb{E}^{n+1}_\varepsilon$.
Since $U_1\subset\mathcal S$, we also have $u(U_1\cap W)\subset P^4_*$, and $u(W\cap V)\subset P^3_V \subset P^4_*$. Thus $u(W)\subset P^4_*$, so $p\in \mathcal S$.

\medskip
\noindent\textbf{Case II(b).} $p\in \partial U_1\cap \partial U_2$, where $U_1,U_2$ are possibly distinct connected components of $M_1$, and $U_1\subset \mathcal S$.
Since $R^\perp\equiv 0$, there exists (after possibly shrinking $W$) a local orthonormal normal frame
$\{{\boldsymbol{e}}_3,{\boldsymbol{e}}_4,\dots,{\boldsymbol{e}}_n\}$ on $W$ with each vector being parallel in the normal bundle.
Because $\nabla^\perp{\boldsymbol{\xi}}_3\equiv 0$ and $|H|\neq 0$, we may arrange, by a constant rotation of the parallel frame, that
\[
{\boldsymbol{e}}_3={\boldsymbol{\xi}}_3\quad\text{on }W.
\]
On $W\cap U_1$, the frame $\{{\boldsymbol{\xi}}_3,{\boldsymbol{\xi}}_4^{(1)},{\boldsymbol{\xi}}_5^{(1)},\dots,{\boldsymbol{\xi}}_n^{(1)}\}$
constructed above is also parallel in the normal bundle, where ${\boldsymbol{\xi}}_5^{(1)},\dots,{\boldsymbol{\xi}}_n^{(1)}$  are actually the constant vectors, and satisfies
\[
A_{{\boldsymbol{\xi}}_\alpha^{(1)}}\equiv 0\quad(\alpha\ge 5).
\]
Hence, on the connected set $W\cap U_1$, the change of frame coefficients between $\{{\boldsymbol{e}}_\alpha\}$ and $\{{\boldsymbol{\xi}}_\alpha^{(1)}\}$
are constant. Up to replacing ${\boldsymbol{e}}_4$ by $-{\boldsymbol{e}}_4$, we may assume
\[
{\boldsymbol{e}}_4={\boldsymbol{\xi}}_4^{(1)}\quad\text{on }W\cap U_1.
\]
In particular, for $\alpha\ge 5$ we have $A_{{\boldsymbol{e}}_\alpha}\equiv 0$ on $W\cap U_1$, because each ${\boldsymbol{e}}_\alpha$ is a constant linear combination
of ${\boldsymbol{\xi}}_5^{(1)},\dots,{\boldsymbol{\xi}}_n^{(1)}$ and $A_{{\boldsymbol{\xi}}_\alpha^{(1)}}\equiv 0$ for $\alpha\ge 5$.

Since ${\boldsymbol{e}}_\alpha\perp{\boldsymbol{e}}_3={\boldsymbol{\xi}}_3$, we have $\mathrm{trace}(A_{{\boldsymbol{e}}_\alpha})=2\langle H,{\boldsymbol{e}}_\alpha\rangle=0$ for $\alpha \geq 
5$. And $\nabla^\perp {\boldsymbol{e}}_\alpha\equiv 0$ on $W$, the associated Hopf-type quadratic differential
\[
\Phi_{{\boldsymbol{e}}_\alpha}:=\inner{ A\left(\frac{\partial u}{\partial z}, \frac{\partial u}{\partial z}\right),{\boldsymbol{e}}_\alpha}\,dz^2
\]
is holomorphic on $W$ by Lemma \ref{lem:holomorphic}. Since $\Phi_{{\boldsymbol{e}}_\alpha}\equiv 0$ on the nonempty open set $W\cap U_1$,
unique continuation implies $\Phi_{{\boldsymbol{e}}_\alpha}\equiv 0$ on $W$, and
\[
A_{{\boldsymbol{e}}_\alpha}\equiv 0\quad\text{on }W,\qquad \alpha =5,\dots,n.
\]
Now restrict to $W\cap U_2$. There, we also have a parallel normal frame
$\{{\boldsymbol{\xi}}_3,{\boldsymbol{\xi}}_4^{(2)},{\boldsymbol{\xi}}_5^{(2)},\dots,{\boldsymbol{\xi}}_n^{(2)}\}$ with
\[
A_{{\boldsymbol{\xi}}_\alpha^{(2)}}\equiv 0\quad(\alpha\ge 5).
\]
Write on $W\cap U_2$
\[
{\boldsymbol{e}}_\alpha=\sum_{\beta =3}^n \tilde c_{\alpha \beta}\,{\boldsymbol{\xi}}_\beta^{(2)},\qquad \alpha=3,\dots,n,
\]
with constant coefficients $\tilde c_{\alpha\beta}$ as a transformation matrix of parallel orthonormal frames. For $\alpha \ge 5$, we have $\tilde{c}_{\alpha 3} = 0$ and 
\[
0=A_{{\boldsymbol{e}}_\alpha}
=\tilde c_{\alpha4}\,A_{{\boldsymbol{\xi}}_4^{(2)}}+\sum_{\beta\ge 5}\tilde c_{\alpha\beta}\,A_{{\boldsymbol{\xi}}_\beta^{(2)}}
=\tilde c_{\alpha 4}\,A_{{\boldsymbol{\xi}}_4^{(2)}}\quad\text{on }W\cap U_2.
\]
Because $U_2\subset M_1$, we have $A_{{\boldsymbol{\xi}}_4^{(2)}}\not\equiv 0$ on $U_2$, hence $\tilde c_{\alpha4}=0$ for all $\alpha\ge 5$. Therefore ${\boldsymbol{e}}_4\in \mathrm{span}\{{\boldsymbol{\xi}}_3,{\boldsymbol{\xi}}_4^{(2)}\}$ on $W\cap U_2$, and since
${\boldsymbol{e}}_4\perp{\boldsymbol{e}}_3={\boldsymbol{\xi}}_3$ we get
\[
{\boldsymbol{e}}_4=\pm{\boldsymbol{\xi}}_4^{(2)}\quad\text{on }W\cap U_2.
\]
Consequently, on $W\cap U_2$,
\[
\mathrm{span}\big\{u,\ du(T\Sigma_0),\ {\boldsymbol{\xi}}_3,\ {\boldsymbol{\xi}}_4^{(2)}\big\}
=
\mathrm{span}\big\{u,\ du(T\Sigma_0),\ {\boldsymbol{e}}_3,\ {\boldsymbol{e}}_4\big\}.
\]
Since $U_1\subset\mathcal S$ and ${\boldsymbol{e}}_4={\boldsymbol{\xi}}_4^{(1)}$ on $W\cap U_1$, we have
$\mathrm{span}\{u,du(T\Sigma_0),{\boldsymbol{e}}_3,{\boldsymbol{e}}_4\}\subset W_*$
on $W\cap U_1$, hence, by connectedness of $W$ and flatness of $D$ as in the proof of \eqref{eq:W_const},
the same inclusion holds on all of $W$:
\[
\mathrm{span}\big\{u(p),\ du(T_p\Sigma_0),\ {\boldsymbol{e}}_3(p),\ {\boldsymbol{e}}_4(p)\big\}\subset W_*\qquad\forall\,p\in W.
\]
In particular, $u(W)\subset \R^n(c)\cap W_*=P^4_*$, so $p\in \mathcal S$.

\medskip
In both cases we conclude $p\in\mathcal S$, hence $\mathcal S$ is closed.
Therefore, since $\Sigma_0$ is connected, $\mathcal S=\Sigma_0$, and we obtain
\[
u(\Sigma_0)\subset P^4_*,
\]
where $P^4_*\cong \R^4(c)$ is totally geodesic in $\R^n(c)$.

If $M_1=\varnothing$, then $\Sigma_0=M_0$ and the above argument adapted on $M_0$ gives $u(\Sigma_0)$ contained in a totally geodesic $P^3\cong \R^3(c)$, hence also in some totally geodesic $P^4$. This completes the proof of Lemma \ref{lem:flat-normal-implies-R4c}.
\end{proof}

Now $u:\Sigma_0\to P^4 = \R^4(c)$ has non-zero parallel mean curvature, constant non-zero contact angle and codimension 2, by Proposition~\ref{lem:PMC-and-angle-in-umbilical-S}. Then, by the same argument as Claim~\ref{claim 1} and Step 1 in the proof of Theorem \ref{main theorem 1} from the previous subsection, it follows that $u(\Sigma_0)$ is contained in a totally umbilic hypersurface $M_2^3\subset P^4$.
Inside the 3-manifold $M^3_2$, by Proposition \ref{lem:PMC-and-angle-in-umbilical-S}, the surface $u(\Sigma)$ has constant mean curvature and a constant non-zero contact angle along each boundary component of $u(\partial \Sigma) \subset \partial B^n \cap P^4$. This proves the second alternative in Theorem \ref{main theorem 2}.  Therefore, we finish the proof of Theorem~\ref{main theorem 2}. 
\end{proof}

\vskip2cm
\bibliographystyle{amsalpha}
\bibliography{references}

@article {chen-chern72,
    AUTHOR = {Chen, Bang-yen},
     TITLE = {On the surface with parallel mean curvature vector},
   JOURNAL = {Indiana Univ. Math. J.},
  FJOURNAL = {Indiana University Mathematics Journal},
    VOLUME = {22},
      YEAR = {1972/73},
     PAGES = {655--666},
      ISSN = {0022-2518,1943-5258},
   MRCLASS = {53A10},
  MRNUMBER = {315606},
MRREVIEWER = {J.\ H.\ White},
       DOI = {10.1512/iumj.1973.22.22053},
       URL = {https://doi.org/10.1512/iumj.1973.22.22053},
}

@article {Ferus1971,
    AUTHOR = {Ferus, Dirk},
     TITLE = {The torsion form of submanifolds in {$E\sp{N}$}},
   JOURNAL = {Math. Ann.},
  FJOURNAL = {Mathematische Annalen},
    VOLUME = {193},
      YEAR = {1971},
     PAGES = {114--120},
      ISSN = {0025-5831,1432-1807},
   MRCLASS = {53.74},
  MRNUMBER = {287493},
MRREVIEWER = {T.\ Okubo},
       DOI = {10.1007/BF02052819},
       URL = {https://doi.org/10.1007/BF02052819},
}

@article {yau1974,
    AUTHOR = {Yau, Shing Tung},
     TITLE = {Submanifolds with constant mean curvature. {I}, {II}},
   JOURNAL = {Amer. J. Math.},
  FJOURNAL = {American Journal of Mathematics},
    VOLUME = {96},
      YEAR = {1974},
     PAGES = {346--366; ibid. {\bf 97 (1975), 76--100}},
      ISSN = {0002-9327,1080-6377},
   MRCLASS = {53C40},
  MRNUMBER = {370443},
MRREVIEWER = {R.\ Osserman},
       DOI = {10.2307/2373638},
       URL = {https://doi.org/10.2307/2373638},
}

@article {Alencar-doCarmo-Tribuzy2010,
    AUTHOR = {Alencar, Hil\'{a}rio and do Carmo, Manfredo Perdig\~{a}o and Tribuzy, Renato},
     TITLE = {A {H}opf theorem for ambient spaces of dimensions higher than
              three},
   JOURNAL = {J. Differential Geom.},
  FJOURNAL = {Journal of Differential Geometry},
    VOLUME = {84},
      YEAR = {2010},
    NUMBER = {1},
     PAGES = {1--17},
      ISSN = {0022-040X,1945-743X},
   MRCLASS = {53C42},
  MRNUMBER = {2629507},
MRREVIEWER = {Fei-Tsen\ Liang},
       URL = {http://projecteuclid.org/euclid.jdg/1271271791},
}

@article {Fetcu2012,
    AUTHOR = {Fetcu, Dorel},
     TITLE = {Surfaces with parallel mean curvature vector in complex space
              forms},
   JOURNAL = {J. Differential Geom.},
  FJOURNAL = {Journal of Differential Geometry},
    VOLUME = {91},
      YEAR = {2012},
    NUMBER = {2},
     PAGES = {215--232},
      ISSN = {0022-040X,1945-743X},
   MRCLASS = {53C42 (53D12)},
  MRNUMBER = {2971287},
MRREVIEWER = {Hui\ Ma},
       URL = {http://projecteuclid.org/euclid.jdg/1344430822},
}

@article {Fetcu-Rosenberg2015,
    AUTHOR = {Fetcu, Dorel and Rosenberg, Harold},
     TITLE = {Surfaces with parallel mean curvature in {S}asakian space
              forms},
   JOURNAL = {Math. Ann.},
  FJOURNAL = {Mathematische Annalen},
    VOLUME = {362},
      YEAR = {2015},
    NUMBER = {1-2},
     PAGES = {501--528},
      ISSN = {0025-5831,1432-1807},
   MRCLASS = {53A10 (53C25 53C42)},
  MRNUMBER = {3359706},
MRREVIEWER = {David\ Blair},
       DOI = {10.1007/s00208-014-1120-9},
       URL = {https://doi.org/10.1007/s00208-014-1120-9},
}

@book{Hoffman-thesis1972,
    AUTHOR = {Hoffman, David Allen},
     TITLE = {Surfaces in constant curvature manifolds with parallel mean curvature vector field},
      NOTE = {Thesis (Ph.D.)--Stanford University},
 PUBLISHER = {ProQuest LLC, Ann Arbor, MI},
      YEAR = {1972},
     PAGES = {64},
   MRCLASS = {99-05},
  MRNUMBER = {2621777},
       URL =
              {http://gateway.proquest.com/openurl?url_ver=Z39.88-2004&rft_val_fmt=info:ofi/fmt:kev:mtx:dissertation&res_dat=xri:pqdiss&rft_dat=xri:pqdiss:7216727},
}

@article {Wong1946,
    AUTHOR = {Wong, Yung-Chow},
     TITLE = {Contributions to the theory of surfaces in a 4-space of
              constant curvature},
   JOURNAL = {Trans. Amer. Math. Soc.},
  FJOURNAL = {Transactions of the American Mathematical Society},
    VOLUME = {59},
      YEAR = {1946},
     PAGES = {467--507},
      ISSN = {0002-9947,1088-6850},
   MRCLASS = {53.0X},
  MRNUMBER = {16231},
MRREVIEWER = {A.\ Lichnerowicz},
       DOI = {10.2307/1990268},
       URL = {https://doi.org/10.2307/1990268},
}

@article {hoffman1973jdg,
    AUTHOR = {Hoffman, David A.},
     TITLE = {Surfaces of constant mean curvature in manifolds of constant
              curvature},
   JOURNAL = {J. Differential Geometry},
  FJOURNAL = {Journal of Differential Geometry},
    VOLUME = {8},
      YEAR = {1973},
     PAGES = {161--176},
      ISSN = {0022-040X,1945-743X},
   MRCLASS = {53C40},
  MRNUMBER = {390973},
MRREVIEWER = {J.\ Vilms},
       URL = {http://projecteuclid.org/euclid.jdg/1214431490},
}

@article {Kenmotsu-Zhou2000,
    AUTHOR = {Kenmotsu, Katsuei and Zhou, Detang},
     TITLE = {The classification of the surfaces with parallel mean
              curvature vector in two-dimensional complex space forms},
   JOURNAL = {Amer. J. Math.},
  FJOURNAL = {American Journal of Mathematics},
    VOLUME = {122},
      YEAR = {2000},
    NUMBER = {2},
     PAGES = {295--317},
      ISSN = {0002-9327,1080-6377},
   MRCLASS = {53C42},
  MRNUMBER = {1749050},
MRREVIEWER = {Hiroo\ Naitoh},
       URL =
              {http://muse.jhu.edu/journals/american_journal_of_mathematics/v122/122.2kenmotsu.pdf},
}

@article {Coburn1939,
    AUTHOR = {Coburn, Nathaniel},
     TITLE = {Surfaces in four-space of constant curvature},
   JOURNAL = {Duke Math. J.},
  FJOURNAL = {Duke Mathematical Journal},
    VOLUME = {5},
      YEAR = {1939},
    NUMBER = {1},
     PAGES = {30--38},
      ISSN = {0012-7094,1547-7398},
   MRCLASS = {99-04},
  MRNUMBER = {1546103},
       DOI = {10.1215/S0012-7094-39-00504-1},
       URL = {https://doi.org/10.1215/S0012-7094-39-00504-1},
}

@Misc{Schouten-Struik1938,
 Author = {Schouten, J. A. and Struik, D. J.},
 Title = {Einf{\"u}hrung in die neueren {Methoden} der {Differentialgeometrie} 2., vollst. umgearb. {Aufl}. {Bd}. 2. {Geometrie}. {Von} {D}. {J}. {Struik}},
 Year = {1938},
 Language = {German},
 HowPublished = {Groningen, {Batavia}: {P}. {Noordhoff} {N}. {V}. {XII}, 338 {S}., 11 {Abb}. (1938).},
 zbMATH = {3030730},
 Zbl = {0019.18301}
}

@article {Fraser-Schoen2015,
    AUTHOR = {Fraser, Ailana and Schoen, Richard},
     TITLE = {Uniqueness theorems for free boundary minimal disks in space
              forms},
   JOURNAL = {Int. Math. Res. Not. IMRN},
  FJOURNAL = {International Mathematics Research Notices. IMRN},
      YEAR = {2015},
    NUMBER = {17},
     PAGES = {8268--8274},
      ISSN = {1073-7928,1687-0247},
   MRCLASS = {53C42},
  MRNUMBER = {3404014},
MRREVIEWER = {Christine\ M.\ Escher},
       DOI = {10.1093/imrn/rnu192},
       URL = {https://doi.org/10.1093/imrn/rnu192},
}

@article {Torralbo-Urbano2012,
    AUTHOR = {Torralbo, Francisco and Urbano, Francisco},
     TITLE = {Surfaces with parallel mean curvature vector in {${\Bbb
              S}^2\times{\Bbb S}^2$} and {${\Bbb H}^2\times{\Bbb H}^2$}},
   JOURNAL = {Trans. Amer. Math. Soc.},
  FJOURNAL = {Transactions of the American Mathematical Society},
    VOLUME = {364},
      YEAR = {2012},
    NUMBER = {2},
     PAGES = {785--813},
      ISSN = {0002-9947,1088-6850},
   MRCLASS = {53A10},
  MRNUMBER = {2846353},
MRREVIEWER = {Hui\ Ma},
       DOI = {10.1090/S0002-9947-2011-05346-8},
       URL = {https://doi.org/10.1090/S0002-9947-2011-05346-8},
}

@article {DeLira-Vitorio2010,
    AUTHOR = {De Lira, Jorge H. S. and Vit\'orio, Feliciano A.},
     TITLE = {Surfaces with constant mean curvature in {R}iemannian
              products},
   JOURNAL = {Q. J. Math.},
  FJOURNAL = {The Quarterly Journal of Mathematics},
    VOLUME = {61},
      YEAR = {2010},
    NUMBER = {1},
     PAGES = {33--41},
      ISSN = {0033-5606,1464-3847},
   MRCLASS = {53C42},
  MRNUMBER = {2592022},
MRREVIEWER = {Henrique\ Fernandes\ de Lima},
       DOI = {10.1093/qmath/han030},
       URL = {https://doi.org/10.1093/qmath/han030},
}

@article {Nitsche1985,
    AUTHOR = {Nitsche, Johannes C. C.},
     TITLE = {Stationary partitioning of convex bodies},
   JOURNAL = {Arch. Rational Mech. Anal.},
  FJOURNAL = {Archive for Rational Mechanics and Analysis},
    VOLUME = {89},
      YEAR = {1985},
    NUMBER = {1},
     PAGES = {1--19},
      ISSN = {0003-9527},
   MRCLASS = {53A10 (52A15)},
  MRNUMBER = {784101},
MRREVIEWER = {Helmut\ Kaul},
       DOI = {10.1007/BF00281743},
       URL = {https://doi.org/10.1007/BF00281743},
}

@article {Souam1997,
    AUTHOR = {Souam, Rabah},
     TITLE = {On stability of stationary hypersurfaces for the partitioning
              problem for balls in space forms},
   JOURNAL = {Math. Z.},
  FJOURNAL = {Mathematische Zeitschrift},
    VOLUME = {224},
      YEAR = {1997},
    NUMBER = {2},
     PAGES = {195--208},
      ISSN = {0025-5874,1432-1823},
   MRCLASS = {53C42 (49Q05 58E15)},
  MRNUMBER = {1431192},
MRREVIEWER = {Jo\~ao\ Lucas Marques Barbosa},
       DOI = {10.1007/PL00004289},
       URL = {https://doi.org/10.1007/PL00004289},
}

@article {Ros-Souam1997,
    AUTHOR = {Ros, Antonio and Souam, Rabah},
     TITLE = {On stability of capillary surfaces in a ball},
   JOURNAL = {Pacific J. Math.},
  FJOURNAL = {Pacific Journal of Mathematics},
    VOLUME = {178},
      YEAR = {1997},
    NUMBER = {2},
     PAGES = {345--361},
      ISSN = {0030-8730,1945-5844},
   MRCLASS = {58E12 (53A10 76D45)},
  MRNUMBER = {1447419},
MRREVIEWER = {Jo\~ao\ Lucas Marques Barbosa},
       DOI = {10.2140/pjm.1997.178.345},
       URL = {https://doi.org/10.2140/pjm.1997.178.345},
}

@book {Spivak-vol4,
    AUTHOR = {Spivak, Michael},
     TITLE = {A comprehensive introduction to differential geometry. {V}ol.
              {IV}},
 PUBLISHER = {Publish or Perish, Inc., Boston, MA},
      YEAR = {1975},
     PAGES = {v+561},
   MRCLASS = {53-02},
  MRNUMBER = {394452},
MRREVIEWER = {N.\ J.\ Hicks},
}

@article{naff2025freeboundarycapillaryminimal,
    title = {Free boundary and capillary minimal surfaces in spherical caps I: Low genus},
    author = {Keaton Naff and Jonathan J. Zhu},
    journal = {arXiv:2512.{12877}},
    year = {2025}
}

@article {Chodosh-Edelen-Li,
    AUTHOR = {Chodosh, Otis and Edelen, Nick and Li, Chao},
     TITLE = {Improved regularity for minimizing capillary hypersurfaces},
   JOURNAL = {Ars Inven. Anal.},
  FJOURNAL = {Ars Inveniendi Analytica},
      YEAR = {2025},
     PAGES = {Paper No. 2, 27},
      ISSN = {2769-8505},
   MRCLASS = {58E12 (49Q10 53C42)},
  MRNUMBER = {4888466},
MRREVIEWER = {Bo\ Yang},
}
\providecommand{\MR}{\relax\ifhmode\unskip\space\fi MR }
\providecommand{\MRhref}[2]{
    \href{http://www.ams.org/mathscinet-getitem?mr=#1}{#2}
}
\end{document}